\documentclass[final]{siamart0516}
\pdfoutput=1

\usepackage{lipsum}
\usepackage{amsfonts}
\usepackage{graphicx}
\usepackage{algorithmic}
\usepackage{subfigure}
\usepackage{amsmath}     
\usepackage{amssymb}
\usepackage{mathtools}
\usepackage{bbm}

\usepackage{color}

\DeclareMathOperator*{\argmin}{arg\,min}

\newcommand{\TheTitle}{{Stationary averaging for multi-scale continuous time Markov chains using parallel replica dynamics} } 
\newcommand{\TheAuthors}{Ting Wang, Petr Plech\'{a}\v{c}, and David Aristoff}

\headers{Parallel replica methods for CTMC}{\TheAuthors}

\title{\TheTitle}
\author{
  Ting Wang\thanks{University of Delaware, Newark, DE, 19716 {(\email{tingw@udel.edu})}},
  \and
  Petr Plech\'{a}\v{c}\thanks{University of Delaware, Newark, DE, 19716 {(\email{plechac@math.udel.edu})}},    
  \and
  David Aristoff\thanks{Colorado State University, Fort Collins, CO, 80523 {(\email{aristoff@rams.colorado.edu})}}.
}

\ifpdf
\hypersetup{
  pdftitle={\TheTitle},
  pdfauthor={\TheAuthors}
}
\fi

\begin{document}

\maketitle

% REQUIRED
\begin{abstract}
We propose two algorithms for 
simulating continuous time Markov chains in the presence 
of metastability. We show that the algorithms correctly estimate, under the ergodicity assumption,
stationary averages of the process. Both algorithms, based on the idea of the parallel replica method,
use parallel computing in order to 
explore metastable sets more 
efficiently. The algorithms require no assumptions 
on the Markov chains beyond ergodicity and the presence 
of identifiable metastability. 
 In particular, there is no
assumption on reversibility.
For simpler illustration of the algorithms, we assume that a synchronous architecture is used
throughout of the paper.
We present 
error analyses, as well as 
numerical simulations 
on multi-scale stochastic
reaction network models in order 
to demonstrate consistency 
of the method and its efficiency.

\end{abstract}

% REQUIRED
\begin{keywords}
Markov chains, Monte Carlo, reversibility, stationary distribution, metastability, parallel replica, stochastic reaction networks, multi-scale dynamics, 
coarse graining 
\end{keywords}

% REQUIRED
\begin{AMS}
60J22, 65C05, 65Z05, 82B31, 92E20
\end{AMS}

%=====================================================

\section{Introduction}
We focus on computing stationary averages of continuous time Markov chains in a synchronous architecture. 
More precisely, if $\pi$ is the stationary distribution of a continuous time Markov chain (CTMC) and 
$f$ is a function on the state space, we aim at estimating the average 
$\pi(f)\equiv \mathbb{E}_\pi[f]$ by taking a time average on a long trajectory of the CTMC. 
There are many methods 
for computing stationary 
averages of stochastic 
processes, however, the 
vast majority of them 
rely on reversibility of 
the process, e.g., as in Markov chain Monte Carlo~\cite{rubinstein}. 
Computational cost of the ergodic (trajectory) averaging becomes prohibitive when the convergence to the stationary
distribution is slow due to metastability of the dynamics, for example in the presence of rare events or large time scale 
disparities (multi-scale dynamics), \cite{schutte}. 
A possible remedy for this issues is to use 
parallel computing in order to accelerate sampling of the state space.
For instance, the parallel tempering method (also known as the replica exchange) 
\cite{geyer1991markov,earl2005parallel,dupuis2012infinite,plattner2011infinite} has been successfully applied to many problems 
by simulating multiple replicas of the original systems,  each replica at a different temperature. 
However, the method requires the time reversibility of the underlying processes, which is typically 
not true for processes that model chemical reaction networks or systems with non-equilibrium steady states.
In fact, 
there are not many methods that parallelize 
Monte Carlo simulation for irreversible processes with metastability, in particular if long-time 
sampling such as ergodic averaging, is required.
We present a parallel computing approach for CTMCs without time reversibility.
One advantage of the proposed algorithms is that they may be used, 
in principle, on arbitrary CTMCs. 
The same idea applies to the continuous state space Markov processes. 
However, gains in efficiency can occur only if the process is metastable.

In this contribution we consider only models described by continuous time Markov chains. 
As a motivating example we study a multi-scale chemical reaction network 
model in which molecules of different types react with different rates
depending on their concentrations and reaction rate constants. 
In this model metastability 
emerges due to the infrequent occurrence of reactions with 
small rates which makes the relaxation to the steady state dynamics extremely slow. 
In the transient regime the finite time distribution can be approximated using the stochastic averaging technique
\cite{weinan2005nested,TTS}, or the tau-leap method \cite{rathinam2003stiffness}. 
However, the former does not apply for stationary distribution estimation
and that the error they introduce on the stationary state is generally difficult to evaluate; 
the latter can be still computationally expensive for long-time simulations.
It is thus desirable to have an efficient algorithm for computing the stationary averages. 
Thus the proposed algorithm will provide a new multi-scale 
simulation method (in particular for stationary averaging estimation) 
for the stochastic reaction networks community.

The presented approach builds on the parallel replica (ParRep) dynamics introduced in the context of
molecular simulations in \cite{voter1998parallel}. The ParRep method used in the context of stochastic differential
equations, e.g. Langevin dynamics, was rigorously analysed in \cite{ParRep-SDE, perez2015parallel}. 
The algorithm we present and analyse builds on the recent work of \cite{ParRep-stationary,ParRep-Chain} where
the ParRep process was studied for discrete-time Markov chains.
%  Recently the  algorithms use 
% parallel processors to 
% explore metastable sets more 
% efficiently than with 
% direct serial simulation. 
% A similar algorithm for 
% discrete time Markov 
% chains~\cite{ParRep-stationary,ParRep-Chain} has recently been proposed. 
% See also~\cite{ParRep-SDE,voter1998parallel} for 
% a related 
% algorithm for simulating 
% stochastic Langevin  
% dynamics and~\cite{perez2015parallel} for a general review.
In our algorithms, 
each time the simulation 
reaches a local equilibrium 
in a metastable set~$W$, 
$R$ independent replicas of the CTMC are launched 
inside the set allowing for parallel simulations of the dynamics at this stage.
The main contribution of this work is a procedure 
for using the replicas  in order
to {\it efficiently} and {\it consistently} 
estimate
the exit time and exit state 
from $W$, along with
the contribution to the 
{\it stationary time average} 
of $f$ from the 
time spent in $W$. We 
emphasize that we are able 
to handle arbitrary functions (or observables)
on the state space, not only those 
that are piece-wise constant, i.e., assuming a single value in each 
$W$.
In the best case, if 
there are $R$ replicas, 
then the simulation leaves
a metastable set about $R$ 
times faster compared to a direct 
serial simulation. 
The consistency of our 
algorithms relies on 
certain properties of the quasi-stationary distribution (QSD) 
which are essentially 
local equilibria associated 
with the metastable sets. 
 
We propose two algorithms 
for computing $\pi(f)$, 
called {\it CTMC ParRep} and 
{\it embedded ParRep}. The 
former uses parallel 
simulation 
of the CTMC, while the 
latter employs parallel simulation 
of its embedded chain, which is 
a discrete time Markov chain 
(DTMC). CTMC ParRep 
(resp. embedded ParRep) relies 
on the fact that, starting at 
the QSD in a metastable set, 
the first time to leave the set 
is an exponential (resp. geometric) random variable
and independent of the exit 
state; see Theorem~\ref{thm:QSD-property} 
below. The algorithms require 
some methods for identifying 
metastable sets, though this 
need not be done a priori -- 
it is sufficient to identify when 
the CTMC is currently 
in a metastable set, 
and when it exits such set. While both 
algorithms can be useful 
for efficient simulation of $\pi(f)$ in 
the presence of metastability, 
we expect the embedded ParRep can 
be significantly more efficient, 
especially when combined with 
a certain type of QSD sampling, 
called Fleming-Viot~\cite{binder2015generalized, blanchet2014theoretical}. 
Though we focus here on 
the computation of $\pi(f)$, 
we note that one of our 
algorithms, 
CTMC ParRep, can be used 
to compute the dynamics 
of the CTMC on a coarse 
space in which each 
metastable set is 
considered a single (meta-)state. 
See the discussion below 
Algorithm~\ref{alg:CTMC-ParRep}.

The advantages of the proposed algorithms include:
(a) no requirement of time reversibility for the underlying dynamics;
(b) they are suitable for long-time sampling;
(c) they may be used, in principle, on arbitrary CTMCs in the presence of metastability.

%This article is organized as  follows. 
In Section~2, 
we briefly review CTMCs 
before defining 
QSDs and detailing 
relevant properties 
thereof. In Section~3, 
we present CTMC ParRep, 
and study how the error 
in the algorithm depends 
on the quality of QSD sampling. 
In Section~4, we present 
embedded ParRep and provide an
analogous error analysis. 
We detail some numerical
experiments on multi-scale
chemical reaction network 
model in
Section~5 in order to demonstrate the consistency and accuracy of the algorithms.

%-----------------------------------------------------------------------------------------
\section{Background and problem formulation}

\subsection{Continuous Time Markov Chains}\label{sec:assume}
Throughout this paper, $X(t)$ is an irreducible and 
positive recurrent continuous time Markov chain (CTMC) with 
values in a countable set~$E$ and $\pi$ denotes the stationary distribution of $X(t)$. We are interested in computing 
stationary averages $\pi(f)$ 
for a bounded
function $f:E \to {\mathbb R}$ by using
the ergodic theorem
\begin{equation}\label{erg}
\lim_{t \to \infty}\frac{1}{t}\int_0^t f(X(s)) ds = \pi(f),
\end{equation}
which holds almost surely for any initial distribution of $X(t)$. The jump times $\tau_n$ 
and holding times $\Delta \tau_n$ 
for $X(t)$ are defined recursively by
\begin{equation*}
\tau_0 = 0, \qquad \tau_n = \inf\{t>\tau_{n-1}:
X(t) \ne X(\tau_{n-1})\}, 
\end{equation*}
and 
\[
\Delta \tau_{n-1} = \tau_{n}-\tau_{n-1} 
\]
for $ n\ge 1$.
We assume that $X(t)$ is non-explosive, 
that is, $\lim_n \tau_n = \infty$ 
almost surely for every initial distribution of $X(t)$. 
This precludes the possibility of  infinitely 
many jumps in finite time.
We denote $X_n = X(\tau_n)$ 
the embedded chain of $X(t)$. 
It is easy to see that 
$X_n$ is a discrete time Markov chain (DTMC).

Recall that $X(t)$ is completely 
determined by its 
infinitesimal generator 
matrix $Q = \{q(x,y)\}_{x,y \in E}$.  
Recall that by convention q(x,x) is chosen such that $\sum_y q(x,y)=0$
and we write $q(x) := -q(x,x)$. 
Note that  
irreducibility implies $q(x)>0$ 
for all $x \in E$. 
It is easy to check 
that $X_n$
has the
transition probability matrix $P = \{p(x,y)\}_{x,y \in E}$ satisfying 
\begin{equation*}
p(x,y) = \begin{cases} \frac{q(x,y)}{q(x)}, & x \ne y, \\ 
                                         0, & x= y 
\end{cases}.
\end{equation*}
We state the following 
well known fact for the later 
reference.

\begin{lemma}\label{lem:conditioned-exp}
For a CTMC $X(t)$ with the corresponding embedded Markov chain $X_n$, the holding time between successive jumps 
$\Delta \tau_0, \Delta \tau_1, \cdots, \Delta \tau_i, \cdots$ are independent conditioned on the embedded chain $X_n$. 
Moreover, $\Delta \tau_i | \{X_n\}$ is exponentially distributed with the rate   
$q(X_i)$
and hence $\mathbb{E}\left[\Delta \tau_i | \{X_n\}\right]= q(X_i)^{-1}$.
\end{lemma}

For details on the above facts, 
see for instance~\cite{baby-Bremaud}.

%--------------------------------------------------------------------------------------------------------
\subsection{The Quasi-stationary Distribution and Metastability}
Below, we write ${\mathbb P}$, 
${\mathbb E}$ 
for various probabilities and 
expectations, the precise meaning of which will be clear from context. We use 
a superscript ${\mathbb P}^\xi$, 
${\mathbb E}^\xi$ to indicate that  
the initial distribution is $\xi$.  
When the initial distribution is $\delta_x$, we write ${\mathbb P}^x$, 
${\mathbb E}^x$. The 
symbol $\sim$ will indicate equality in probability law. $\textup{Re}(\cdot)$ and $|\cdot|$
denote the real part and modulus of a complex number.

Our ParRep algorithms rely on 
certain properties of quasi-stationary distributions, 
which we now briefly review.
Let $W\subset E$ be fixed and consider the first exit time of $X(t)$ from $W$, that is,
\[
T = \inf\{t > 0; X(t) \notin W\}.
\]
We consider also the first exit time of 
$X_n$ from $W$,
\[ 
N = \inf\{n>0; X_n \notin W\}.
\]
A quasi-stationary distribution (QSD) of $X(t)$ in $W$ (or $X_n$ in $W$) is defined as follows.
\begin{definition}
A probability distribution $\nu$ with support in $W$ is a quasi-stationary distribution 
for $X(t)$ in $W$ if for each $y \in W$ and $t>0$, 
\begin{equation}\label{eqn:QSD-CTMC}
\nu(y) = \mathbb{P}^{\nu}(X(t) = y\,|\,T>t).
\end{equation} 
Similarly, a probability distribution $\mu$ with support in $W$ is a QSD for $X_n$ in $W$ if for each $y \in W$ and  $n>0$, 
\begin{equation}\label{eqn:QSD-DTMC}
\mu(y) = \mathbb{P}^{\mu}(X_n = y \,|\,N>n).
\end{equation}
\end{definition}
Throughout we write $\nu$ for a QSD of the CTMC $X(t)$ and $\mu$ for a QSD of the embedded chain $X_n$. The associated 
set $W$ will be implicit since no 
ambiguities should arise. We will write
\begin{equation}\label{eqn:CTMC-Yaglom}
\nu_t(A) = \mathbb{P}^x(X(t) \in A\,|\, T>t) 
\end{equation}
for the distribution of $X(t)$ 
conditioned on $T>t$, and 
\begin{equation}\label{eqn:DTMC-Yaglom}
\mu_n(A) = \mathbb{P}^x(X_n \in A\,|\, N>n).
\end{equation}
for the distribution of $X_n$ conditioned on $N>n$. Notice we do not 
make explicit the dependence on the starting point $x$.

We summarize existence, uniqueness, 
and convergence properties of the QSD in Theorem~\ref{thm:QSD-spectral} below (see \cite{QSD,non-negativeMC}). In Theorem~\ref{thm:QSD-spectral} below, for simpler presentation, we assume $W$ is finite. That 
allows us to characterize 
convergence to the QSD of $X(t)$ and $X_n$  in terms of spectral properties of their 
generator and transition matrices.
We emphasize, however, there exists more general results to guarantee the convergence to the QSD and hence the finiteness of $W$ is not necessary for consistency of the algorithms proposed in this paper.

Recall that $Q$ is the infinitesimal generator matrix of $X(t)$ and $P$ is the transition probability matrix of the DTMC $X_n$. 
We denote $Q_W = \{q_{xy}\}_{x,y \in W}$ and $P_W = \{p_{xy}\}_{x,y \in W}$ the restrictions of $P$ and $Q$ to $W$.

\begin{theorem}\label{thm:QSD-spectral}
Let $W$ be finite and nonabsorbing 
for $X(t)$, and assume $P_W$ 
is irreducible. 
\begin{itemize}
\item [(a)] 
The eigenvalues $\lambda_1,\lambda_2,\ldots$ of $Q_W$ can be ordered so that
\[0> \lambda_1 > \textup{Re}(\lambda_2)\geq \ldots,\]
where $\lambda_1$ has the left eigenvector $\nu$ which is a probability distribution on $W$.
Moreover, $\nu$ is the unique quasi-stationary distribution of $X(t)$ in $W$, and for all $x,y \in W$,
\begin{equation}\label{eqn:rate-QSD-CTMC}
\left|\nu_t(y) - \nu(y)\right|  = 
\left|{\mathbb P}^x(X(t) = y|T>t) - 
\nu(y)\right|
\le C(x)e^{-(\lambda_1-\beta) t},
\end{equation}
with $C(x)$ a constant depending on $x$, and $\beta$ any real number satisfying
$\textup{Re}(\lambda_2) < \beta < \lambda_1$. 

\item[(b)] Suppose $P_W$ is also 
aperiodic. Then the eigenvalues $\sigma_1,\sigma_2,\ldots$ of $P_W$ can be ordered so that 
\[
1>\sigma_1 > |\sigma_2|\geq \ldots,
\]
where $\sigma_1$ has the left eigenvector $\mu$ which is a probability distribution on $W$. 
Moreover, $\mu$ is the unique quasi-stationary distribution of $X_n$ in $W$ and for all $x,y \in W$,
\begin{equation}\label{eqn:rate-QSD-DTMC}
\left|\mu_n(y) - \mu(y)\right| 
= \left|{\mathbb P}^x(X_n = y|N>n) - \mu(y)\right|
\le D(x)\left(\frac{\gamma}{\sigma_1}\right)^n,
\end{equation}
with $D(x)$ a constant depending on $x$, 
and $\gamma$ any real number satisfying $|\sigma_2| < \gamma < \sigma_1$.
\end{itemize}
\end{theorem}

\begin{proof}
%We only prove part (a), since the arguments for (b) are similar. 
We first justify the expression for the eigenvalues. 
Observe that for $x \ne y \in W$, 
we have $q(x,y) >0$ 
if and only if $p(x,y) >0$. 
It follows that $Q_W$ is 
irreducible if and only if 
$P_W$ is irreducible; see 
Definition 2.1 in~\cite{non-negativeMC}. Now 
let $I$ be the all ones column 
vector, $I(x) = 1$ for $x \in W$. 
Recall that $q(x,y) \ge 0$ for 
every $x \ne y \in E$ and 
$\sum_y q(x,y) = 0$ for every $x \in E$.  
This implies that $Q_W I \le 0$ 
component-wise. 
Since 
$W$ is non-absorbing, for 
some $x \in W$ and $y \notin W$ we have 
$q(x,y) > 0$, and it follows 
that $\sum_{z \in W} q(x,z) < 0$. This shows that at least one 
component of $Q_W I$ is 
strictly negative. The expression 
for the eigenvalues, and 
the fact that $\nu$ is signed (hence 
a probability distribution, after 
normalization) 
now follows from 
Theorem 2.6 of Seneta~\cite{non-negativeMC}.

To see $\nu$ is the QSD for $X(t)$ in $W$, we define the stopped process $X^T(t) = X(t \wedge T)$ such that $X(t)$ is 
absorbed outside $W$. 
For any $x,z \in E$, let $I_x$ be 
the column vector 
$I_x(z) = 1$ if $x = z$
and $I_x(z) = 0$ otherwise. 
Finiteness of $W$ ensures 
that ${\mathbb P}^x(X^T(t) = y) = I_x^{\prime} e^{Q_W t}I_y$. Thus, for each $y \in W$, 
\[
\mathbb{P}^{\nu}(X(t) = y, T > t)=\mathbb{P}^{\nu}(X^T(t) = y) = 
\nu e^{Q_W t} I_y =  e^{\lambda_1t}\nu(y)
\]
and
\[
\mathbb{P}^{\nu}(T > t) = \mathbb{P}^{\nu}(X^T(t) \in W)=e^{\lambda_1 t},
\]
which leads to $\nu(y) = \mathbb{P}^{\nu}(X(t) = y|T>t)$.

Now we turn to the convergence to $\nu$.  It follows from Theorem~2.7 in~\cite{non-negativeMC} that there is a constant 
$C(x)$ depending on $x$ such 
that for any real $\beta$ with $\textup{Re}(\lambda_2) < \beta$, 
\begin{equation}\label{eqn:PF-docomposition}
\mathbb{P}^x(X(t) = y, T>t)  = \mathbb{P}^x(X^T(t) = y)=C(x) e^{\lambda_1 t}\nu(y) + \mathcal{O}(e^{\beta t})
\end{equation}
and
\begin{equation}
\mathbb{P}^x(T>t) = C(x) e^{\lambda_1 t} + \mathcal{O}(e^{\beta t}),
\end{equation} 
It follows that
\[
|\nu_t(t)-\nu(y)| = |\mathbb{P}^x(X(t) = y\,|\, T>t) 
- \nu(y)| \leq C(x) e^{-(\lambda_1-\beta)t}
\]
where $C(x)$ is now a (possibly different) constant depending on $x$.

The arguments in (b) are similar, using the Perron-Frobenius theorem (Seneta~\cite[Theorem 1.1]{non-negativeMC}).

\end{proof}

For analogous results on the QSD in more general settings, see~\cite[Theorem4.5]{QSD} for CTMCs and \cite[Theorem 1]{del2004particle} 
for DTMCs. 
We are now ready to define metastability.
\begin{definition}\label{def:ms}
Let $W$ and $\lambda_i,\sigma_i$ be as in Theorem~\ref{thm:QSD-spectral}.
\begin{itemize}
\item[1.] 
$W$ is metastable for $X(t)$ if $\lambda_1 \approx 0$ and
\begin{equation}\label{eqn:metastable-CTMC}
|\lambda_1| \ll |\lambda_1-\textup{Re}(\lambda_2)|.
\end{equation}
$X(t)$ is metastable if it has at least 
one metastable set $W$.
\item[2.] $W$ is metastable for $X_n$ if 
$\sigma_1 \approx 1$ and
\begin{equation}\label{eqn:metastable-DTMC}
\sigma_1 \gg \frac{|\sigma_2|}{\sigma_1}.
\end{equation}
$X_n$ is metastable if it has at least 
one metastable set $W$.
\end{itemize}
\end{definition}
In light of Theorem~\ref{thm:QSD-spectral}, Conditions 1-2 in Definition~\ref{def:ms} essentially say that the time to leave $W$ is large 
in an absolute sense, and 
the time to leave $W$ is large relative to the time to converge to the QSD in $W$.
Metastability of the CTMC is not necessarily equivalent to the metastability of its underlying embedded chain, as we now show.
Consider $X(t)$ with the infinitesimal 
generator
\begin{equation*}
Q = \begin{pmatrix} 
      -1 & 1/2 & 1/2 & 0 \\ 
     1/2 & -1 & 1/2 & 0 \\ 
     0   & \epsilon/2 & -\epsilon & \epsilon/2 \\
     0   & 0 & 1 & -1
    \end{pmatrix}, 
\end{equation*}
where $\epsilon \approx 0$ is positive.
Then $W = \{1,2,3\}$ is metastable 
for $X(t)$ but not for $X_n$, since 
\begin{equation*}
\sigma_1 \approx 0.81,\quad 
|\sigma_2| \approx 1/2, \qquad 
\lambda_1 \approx -\epsilon/2, \quad 
\textup{Re}(\lambda_2) \approx -1/2.
\end{equation*}
Now consider $X(t)$ with the infinitesimal generator
\begin{equation*}
Q = \begin{pmatrix} 
    -\epsilon^{-1}  & \epsilon^{-1}/2 & \epsilon^{-1}/2 & 0 \\ 
     \epsilon^{-1}-1& -\epsilon^{-1} & 1  & 0 \\ 
                  0 & \epsilon^{-1}-1& -\epsilon^{-1} & 1 \\
                  0 & 0 & 1 & -1
    \end{pmatrix}.
\end{equation*}
Then $W = \{1,2,3\}$ is 
metastable for $X_n$
but not for $X(t)$, 
since 
\begin{equation*}
\sigma_1 \approx 1-\epsilon/5, \quad 
|\sigma_2| \approx \sqrt{2}/2, \qquad 
\lambda_1 \approx -1/5, \quad 
\textup{Re}(\lambda_2) \approx -3\epsilon^{-1}/2.
\end{equation*}

Algorithm~\ref{alg:CTMC-ParRep} below
requires a collection of 
metastable sets for $X(t)$, 
and Algorithm~\ref{alg:embedded-ParRep} 
requires a collection of 
metastable sets for $X_n$. 
The only assumption we make on 
these sets is that they are 
pairwise disjoint. (The sets may be different 
for the two algorithms, as noted above.)
Throughout we write $W$ to denote a
generic metastable set. 
We emphasize that {\em we do not assume the
metastable sets form a partition 
of $E$}: the union of the 
metastable sets may be a proper 
subset of $E$. 
Here and below, 
we assume that each $W$ has 
a unique QSD and that 
$\nu_t$ (and $\mu_t$) converge 
to the QSD in total variation norm, 
for any starting point $x$.
Recall that this 
is true under the assumptions of Theorem 3

We conclude this section 
by mentioning properties 
of the QSD which are essential 
for the consistency 
of our algorithms 
in Section~3 and~4 below. 

\begin{theorem}\label{thm:QSD-property}
~
\begin{itemize}
\item[1.] Suppose $X(0) \sim \nu$. Then $T$ is exponentially distributed with the parameter
$-\lambda_1$: 
\[\mathbb{P}^{\nu}(T>t) = e^{\lambda_1 t}, 
\qquad t>0,\]
and $T$ and $X(T)$ are independent.
\item[2.] Suppose $X_0 \sim \mu$. Then $N$ is geometrically distributed with the parameter $1- \sigma_1$:
\[
\mathbb{P}^{\mu}(N > n) = \sigma_1^n, 
\qquad n=1,2,\ldots,
\]
and $N$ and $X_N$ are independent.
\end{itemize}
\end{theorem}

\begin{proof}The first part of 1 and 2 was shown in Theorem~\ref{thm:QSD-spectral}. For 
the rest of the proof 
see~\cite{QSD}.
\end{proof}

%-------------------------------------------------------------------------------------------------------
\section{The CTMC ParRep Method}
\subsection{Formulation of the CTMC Algorithm}\label{sec:CTMC}

In this section, we introduce a 
method for accelerating the computation 
of $\pi(f)$, where we recall 
$f: E \to \mathbb{R}$ is any bounded function and $\pi$ is the stationary distribution. 
We call this algorithm CTMC ParRep, 
for reasons that will be outlined below.
Before we describe CTMC ParRep, we introduce some notation. 
Throughout, $X^1(t), \ldots, X^R(t)$ 
will be independent processes with 
the same law as $X(t)$ and 
with initial distributions supported in $W$. Recall that the first exit time 
of $X(t)$ from $W$ is 
\begin{equation*}
T = \inf\{t>0:X(t) \notin W\}.
\end{equation*}
Similarly, for $r=1,\ldots,R$, we define the 
first exit time of $X^r(t)$ from $W$ 
by
\[
T^r =\inf\{t>0: X^r(t) \notin W\}
\] 
and the smallest one among them by
 \[
 T^*=\min_r T^r.
 \]
We denote the index of the replica with the
first exit time $T^*$ by $M$, i.e.,
\[
M = \argmin_r T^r.
\] 
$T$, $T^r$, $T^*$ and $M$ 
depend on $W$, but we do not 
make this explicit. 
 
We are in the position to present the CTMC ParRep in Algorithm~\ref{alg:CTMC-ParRep}.
In this algorithm, we will need user-chosen parameters 
$t_c$ associated with 
each metastable set $W$. Roughly speaking, these parameters correspond to the 
time for $X(t)$ to converge to the QSD in $W$. 
The accumulated value $F(f)_{\mathrm{sim}}$ serves as a quantity that approximates the integral
$\int_0^{T_{\mathrm{end}}} f(X(s))\,ds$ when the algorithm terminates.
Note that at the end of the algorithm we often have $T_{\mathrm{sim}} \geq T_{\mathrm{end}}$ since the ParRep process could reach $T_{\mathrm{end}}$ during the parallel stage. However, this is not an issue as long as $T_{\mathrm{sim}}$ is large enough at the end of the algorithm so that the time average is well approximated.

\begin{algorithm}[ht!]
\begin{algorithmic}[1]
\caption{CTMC ParRep}
\label{alg:CTMC-ParRep}
\STATE
Set a decorrelation threshold $t_c$ for each metastable set $W$. Initialize the 
simulation time clock $T_{\mathrm{sim}}=0$ and the accumulated value $F(f)_{\mathrm{sim}}=0$. We 
will write $X^{\mathrm{par}}(t)$ for a 
simulation process that obeys the 
law of $X(t)$. A complete ParRep cycle consists of three stages. 

%---------------------------------------------------------------------------------------
\STATE
{\bf Decorrelation Stage} :
    Starting at $t = T_{\mathrm{sim}}$, evolve $X^{\mathrm{par}}(t)$ until it spends an interval of the time length $t_c$ 
    inside the same metastable set $W$. That 
    is, evolve $X^{\mathrm{par}}(t)$ from 
    time $t=T_{\mathrm{sim}}$ until time
\begin{equation*}
T_{\mathrm{corr}} = \inf\{t\ge T_{\mathrm{sim}}+t_c: X^{\mathrm{par}}(s) \in W~\mathrm{for~all}~s\in [t-t_c, t]~\mathrm{for~some}~W\}.
\end{equation*}
Then update 
\begin{equation*}
F(f)_{\mathrm{sim}} = F(f)_{\mathrm{sim}}+ \int_{T_{\mathrm{sim}}}^{T_{\mathrm{corr}}} f(X^{\mathrm{par}}(t))\,dt,
\end{equation*} 
set $T_{\mathrm{sim}} = T_{\mathrm{corr}}$, 
and proceed to the dephasing stage.

%---------------------------------------------------------------------------------------
\STATE
{\bf Dephasing Stage} : Let $W$ be 
such that $X^{\mathrm{par}}(T_{\mathrm{sim}}) \in W$, 
that is, $W$ is the metastable set from the end of 
the last decorrelation stage. 
Generate $R$ independent samples $x_1, \ldots, x_R$ from $\nu$, the QSD of $X(t)$ in $W$. 

Then proceed to the parallel stage.  
%---------------------------------------------------------------------------------------
\STATE
{\bf Parallel Stage} : Start $R$ parallel processes $X^1(t), \ldots, X^R(t)$ 
at $x_1, \ldots,x_R$, and evolve them simultaneously from time $t=0$ until time $T^* = \min_r T^r$.
Here all $R$ processes are simulated in parallel.
Then update
\begin{equation}\label{eqn:CTMC-parallel-stage}
\begin{split}
     &F(f)_{\mathrm{sim}} = F(f)_{\mathrm{sim}} + \sum_{r=1}^{R} \int_{0}^{T^*} f(X^r(s)) ds,\\
     &T_{\mathrm{sim}}= T_{\mathrm{sim}}+RT^*,
\end{split}
\end{equation}
set $X^{\mathrm{par}}(T_{\mathrm{sim}}) = X^{M}(T^*)$, and return to the decorrelation stage. 
%---------------------------------------------------------------------------------------
\STATE
The algorithm is stopped when $T_{\mathrm{sim}}$ reached 
a user-chosen terminal time $T_{\mathrm{end}}$. 

The stationary average $\pi(f)$ is estimated as 
\begin{equation*}\label{erg_approx}
\pi(f) \approx F(f)_{\mathrm{sim}}/T_{\mathrm{sim}}. 
\end{equation*}
\end{algorithmic}
\end{algorithm}

%We make some remarks before we present the error analysis in the next section. 
%\begin{itemize}
%\item
%The CTMC $X(t)$ 
%can be simulated by the classical stochastic simulation algorithm (SSA) (see \cite{gillespie1977exact} for reference). 
%Essentially, one samples the 
%holding times $\Delta \tau_n$ and the next state of the embedded chain $X_{n+1}$ 
%at the same time, using two random numbers. The process $X(t)$ can 
%then be reconstructed using 
%\begin{equation*}
%X(t) = X_n, \qquad \Delta \tau_0+\ldots+\Delta \tau_n \le t < \Delta \tau_0+\ldots+\Delta\tau_{n+1}.
%\end{equation*}
%This is how we simulate $X^{\mathrm{par}}(t)$ and $X^1(t),\ldots,X^R(t)$ in the decorrelation and parallel stages. 

If $X^{\mathrm{par}}(t)$ remains in $W$ for sufficiently long time (i.e., decorrelation threshold $t_c$), it is distributed nearly 
according to the QSD $\nu$ of $X(t)$ in $W$ by Theorem~\ref{thm:QSD-spectral}. This means 
that at the end of the decorrelation 
stage, $X^{\mathrm{par}}(T_{\mathrm{sim}})$ can 
be considered a sample of $\nu$.

The aim of the dephasing stage is to prepare a sequence of independent initial states with distribution $\nu$. 
There are several ways for achieving this. 
Perhaps the simplest is the rejection method. 
In this procedure, each of the $R$ replicas evolves independently. 
A parameter $t_p$ similar to the decorrelation threshold $t_c$ is selected. 
If a replica leaves $W$ before spending 
a time interval of length $t_p$ 
in $W$, it restarts in $W$ from the original initial state. 
Once all the replicas remain in $W$ for time $t_p$, 
we stop and take $x_1,\ldots,x_R$ as 
the final states of all the replicas in the dephasing stage and use them for the subsequent 
parallel stage.
Besides rejection sampling, another method is a Fleming-Viot 
based particle sampler; see 
the discussion after
Algorithm~\ref{alg:embedded-ParRep}
below.
Finally, we comment that  we can reduce the overhead related to the dephasing stage by starting the dephasing stage immediately (instead of waiting for decorrelation stage to finish) when the $X^{\text{Par}}(t)$ enters into a new metastable set \cite{binder2015generalized}.

The acceleration of CTMC ParRep comes from the parallel stage. 
Recall that, for each $r=1,\ldots,R$, if $x_1,\ldots,x_R$ are independent, identically distributed (iid) with the
common distribution $\nu$, then $T^1,\ldots,T_R$ are independent exponential random variables with common parameter $\lambda_1$. 
Using $T^* = \min_{r} T^r$, it is then 
easy to check that $RT^*$ has the same distribution
as $T^1$. See Lemma~\ref{lem:exp} below.  
This means one only needs to wait for $T^*$
instead of $T^1$ to observe an 
exit from $W$. Note 
that this is true whether or 
not $W$ is metastable, so 
efficiency of the parallel stage does 
not require metastability. 
However, the dephasing stage
is not efficient if $W$ is 
not metastable. That is because, 
in practice, the samples 
$x_1,\ldots,x_R$ are obtained 
by simulating trajectories 
which remain in $W$ for 
a sufficiently long time $t_p$. 
Such samples are hard to 
obtain when the typical time $t_p$ 
for $x_1,\ldots,x_R$ 
to reach the QSD in $W$ is 
not much smaller than the 
typical time to leave $W$.

To see that each parallel stage has a consistent contribution to the stationary average, we make the following two observations. Suppose that $x_1,\ldots,x_R$ 
are iid samples from $\nu$.
\begin{itemize}
\item [1.] The joint law of $(RT^*, X^M(T^*))$ is the same as that of $(T^1, X(T^1))$. That is, the joint distribution of the first exit time and the exit state in the parallel stage is independent of the number of replicas.
\item [2.] The expected value of $\sum_{r=1}^{R} \int_{0}^{T^*} f(X^r(s)) ds$ in \eqref{eqn:CTMC-parallel-stage} is the same as that of $\int_0^{T^1} f(X^1(s)) ds$.
That is, the expected contribution to $F(f)_{\mathrm{sim}}$ from each parallel stage is independent of the number of replicas.
\end{itemize}
The first observation is a consequence of the Theorem~\ref{thm:QSD-property}, and the second will be proved in Theorem~\ref{thm:CTMC-parallel-unbiased} 
below. Consistency of stationary 
averages follows from the points 1-2 above and 
the law of large numbers. Since 
there are indefinitely many parallel stages in a given $W$, consistency is ensured as long as the
expected contribution to $F(f)_{\mathrm{sim}}$ 
from the parallel stage has 
the correct expected value.
See~\cite{ParRep-stationary} for 
details and discussion in a related 
discrete time version of the algorithm under some idealized assumptions.

%It isn't necessary that the contribution 
%to $F(f)_{\mathrm{sim}}$ from the parallel 
%stage has the same {\em law} for any 
%number of replicas -- though this turns out to be true for CTMC ParRep. 

%\end{itemize}

The CTMC ParRep algorithm suffers some 
serious drawbacks. Even if the 
parallel 
processors are synchronous, $M$ and $T^*$ may 
not be known at the wall clock time 
when the first 
replica leaves $W$. 
The reason is that the holding times for a CTMC are random, while the wall clock time for simulating each jump of the CTMC is always roughly the same.
We illustrate this problem in Figure~\ref{fig:nonuniformity}.
\begin{figure}[h]
\centering
\includegraphics[scale=0.8]{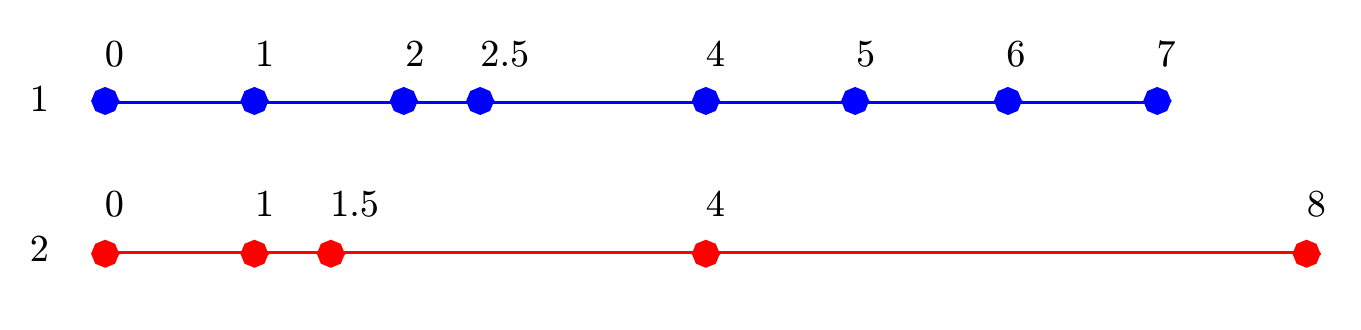}
\caption{The parallel stage of the CTMC ParRep algorithm with two replicas. $R1$ escapes from $W$ at $t = 7$ with $7$ transitions while $R2$ escapes at $t=8$ but with only $4$ transitions. In the parallel stage of the CTMC ParRep algorithm, $R2$ escaped from $W$ before $R1$ does but $T^2 > T^1$. There is no acceleration in this case since the parallel stage does not terminate when $R2$ escapes.}
\label{fig:nonuniformity}
\end{figure}
In the worst possible case, in order to determine $M$ and $T^*$, we 
must wait for all the replicas to leave $W$. 
However, one can set a variable $T_{\mathrm{min}}$ to record the current minimum first exit time over all replicas which have left $W$, and terminate any replicas which reach time $T_{\mathrm{min}}$ but have not left $W$, since no 
replica contributes to the accumulated 
value past time $T_{\mathrm{min}}$.  
Since the expected first exit times $\mathbb{E}[T^r], r=1,\ldots, R$ are roughly the same, if the variance in the number of jumps of $X^r(t)$ before 
time $T^*$ is small for all $r=1,\ldots,R$, then we can expect that the parallel stage stops after only a few replicas leave $W$.

For the same reason, there is another major drawback of CTMC ParRep. If $f$ takes multiple values in $W$, then the computation of 
$\sum_{r=1}^{R} \int_0^{T^*} f(X^r(s)) ds$ in \eqref{eqn:CTMC-parallel-stage} requires storing the entire history of the value of f on each replica in that parallel stage. 
We illustrate this drawback by considering the case of two replicas $r_1$ and $r_2$ with first exit time $T^1$ and $T^2$, respectively.
Suppose $T^1 < T^2$ and hence $T^* =  T^1$. Let us assume that in terms of wall clock time,  $r_2$ exits $W$ before $r_1$ does.
At the end of the parallel stage (i.e., after both of $T^1$ and $T^2$ are sampled ) we have the running sum $\int_0^{T^2} f(X^2(s)) ds$ from $r_2$.
However,  by the CTMC ParRep algorithm,  we only need $\int_0^{T^*} f(X^2(s)) ds$ indeed. 
If we only keep track of the running sum, we are unable to recover $\int_0^{T^*} f(X^2(s)) ds$ from $T^2$. 
Hence, the implementation of the CTMC ParRep might be memory demanding unless one is interested in the equilibrium average 
of a metastable-set invariant function $f$, i.e., if $f(x)$ has only one value in each metastable set $W$. In 
Section~4 we present 
another algorithm, called embedded ParRep, which addresses these drawbacks.

%--------------------------------------------------------------------------------------------------------------

%---------------------------------------------------------------------------------------
\subsection{Error Analysis of CTMC ParRep}

Here and below 
we will write 
${\mathbb E}^{\nu^R}$ for the 
expectation of $(X^1(t),\ldots,X^R(t))$ 
starting at $\nu^R$, where for
\begin{equation*}
\nu^R(x_1,\ldots,x_R) = \prod_{r=1}^R \nu(x_r), \qquad x_1,\ldots,x_R \in W.
\end{equation*}
We begin with a simple well known lemma.
\begin{lemma}\label{lem:exp}
Suppose $T^1,\ldots,T^R$ are iid exponential
random variables with the parameter 
$\lambda_1$. Then $T^* = \min_{1 \le r \le R} T^r$ is exponentially distributed 
with the parameter $R\lambda_1$. In particular, 
$RT^*$ has the same distribution as $T^1$.
\end{lemma}

We now show that if the 
dephasing sampling is exact, then 
the expected contribution to the accumulated 
value $F(f)_{\mathrm{sim}}$ from the parallel 
step of Algorithm~\ref{alg:CTMC-ParRep} 
is exact. 
\begin{theorem}\label{thm:CTMC-parallel-unbiased}
Suppose in the dephasing step
$(x_1,\ldots,x_R) \sim \nu^R$. Then the expected contribution to $F(f)_{\mathrm{sim}}$ from the parallel stage 
of Algorithm~\ref{alg:CTMC-ParRep} is 
independent of the number of replicas,
\[\mathbb{E}^{\nu^R}\left[\sum_{r=1}^{R}\int_0^{T^*} f(X^r(s)) ds\right] = 
{\mathbb E}^\nu\left[\int_0^T f(X(s))ds\right] = \nu(f)\mathbb{E}^{\nu}[T].\] 
%That is, the averaged accumulated reward from each parallel stage is independent of the number of replicas. 
\end{theorem}
\begin{proof}
First we consider the case with a single replica. 
We condition on the exit time $T^1$ and write
\[
\mathbb{E}^{\nu}\left[\int_0^{T^1} f(X^1(s)) ds\right] = \int_0^{\infty}\mathbb{E}^{\nu}\left[\left.\int_0^t f(X^1(s)) ds \right| T^1 = t\right] \mathbb{P}^{\nu}(T^1 \in dt).
\]
Interchanging the two integrals of the right-hand side leads to
\[
\int_0^{\infty}\int_s^{\infty}\mathbb{E}^{\nu}\left[ f(X^1(s)) | T^1 = t\right] \mathbb{P}(T^1 \in dt) ds.
\]
Note that the inner integral can be written as $\mathbb{E}^{\nu}\left[f(X^1(s)) \mathbbm{1}_{T^1 >s}\right]$ and hence 
\[
\mathbb{E}^{\nu}\left[\int_0^{T^1} f(X^1(s)) ds\right]=\int_0^{\infty} \mathbb{E}^{\nu}\left[f(X^1(s))| T^1 >s\right]\mathbb{P}^{\nu}(T^1 > s) ds.
\]
Owing to the definition of QSD and the fact that $\mathbb{E}^{\nu}[T^1] = \int_0^{\infty} \mathbb{P}^{\nu}(T^1 >s) ds$, 
\[
\mathbb{E}^{\nu}\left[\int_0^{T^1} f(X^1(s)) ds\right]=\nu(f) \mathbb{E}^{\nu}[T^1].
\]
In the case of multiple replicas, similar steps can be used to show that 
\[
\sum_{r=1}^{R}\mathbb{E}^{\nu^R}\left[\int_0^{T^*} f(X^r(s)) ds\right]=\sum_{r=1}^{R}\int_0^{\infty} \mathbb{E}^{\nu^R}\left[f(X^r(s))| T^* >s\right]\mathbb{P}^{\nu^R}(T^* > s) ds.
\] 
Recall that $T^* > s$ if and only if $T^r >s$ for all $r=1, \ldots, R$. 
Using this, the fact that $T^1,\ldots,T^r$ are independent, and the definition of the QSD, we get 
\[ 
\mathbb{E}^{\nu^R}\left[f(X^r(s))| T^* >s\right] =  \mathbb{E}^{\nu}\left[f(X^r(s))| T^r >s\right] = \nu(f).
\] 
Thus 
\[
\sum_{r=1}^{R}\mathbb{E}^{\nu^R}\left[\int_0^{T^*} f(X^r(s)) ds\right]
= \nu(f)\sum_{r=1}^R \int_0^\infty {\mathbb P}^{\nu^R}(T^*>s)ds = \nu(f)R{\mathbb E}^{\nu^R}[T^*]. 
\]
Finally, the result follows from Lemma~\ref{lem:exp}.
%\[\nu(f)R{\mathbb E}^{\nu^R}[T^*]= \nu(f){\mathbb E}^\nu[T].\]
\end{proof}

%In practice, the QSD can never be sampled exactly and hence there is an error associated with each parallel stage. 
The purpose of CTMC ParRep is to efficiently simulate very long trajectories of a metastable CTMC and estimate the equilibrium average $\pi(f)$. 
CTMC ParRep can produce 
accelerated dynamics of the CTMC 
on a coarse state space where 
each coarse set corresponds to some 
$W$; see the discussion below
Algorithm~\ref{alg:embedded-ParRep} below.
Our numerical experiments suggest that CTMC ParRep (and also embedded ParRep described below) are consistent 
for estimating the stationary distribution.

For CTMC ParRep, we justify this claim
in Theorem~\ref{thm:CTMC-F} below, 
which shows that, starting 
in some $W$ and waiting until the 
simulation leaves $W$, the error for a complete ParRep cycle 
in CTMC ParRep compared to 
direct (serial) simulation vanishes 
as $t_c$ increases. See 
Theorem~\ref{thm:embedded-F} below 
for the analogous result on embedded 
ParRep. 
We note that each ParRep cycle 
produce an error in the estimation of 
stationary averages that does not disappear as $T_{\mathrm{sim}} \to \infty$.  
However, we expect that the error vanishes as the thresholds $t_c = t_p \to \infty$. 
Study of the this error is more involved and will be the focus of another work. 

Recall we have assumed 
convergence of $\|\nu_{t_c}-\nu\|_{\mathrm{TV}} \to 0$ as $t_c \to \infty$, for 
every starting point $x \in E$, where 
$\|\cdot\|_{\mathrm{TV}}$ denotes total variation 
norm. See for instance~Theorem~\ref{thm:QSD-spectral} for conditions guaranteeing 
this convergence.

%Let $T^{\mathrm{par}}$ be the first exit time from $W$ for $X^{\mathrm{par}}(t)$ in 
%this stage, with 
%$T^{\mathrm{par}} = \infty$ if 
%we proceed to the dephasing 
%stage before $X^{\mathrm{par}}(t)$  
%leaves $W$. 

\begin{theorem}\label{thm:CTMC-F}
Consider CTMC ParRep 
starting at $x \in W$ in the 
decorrelation stage.
Assume the dephasing stage sampling is exact, that is, $(x_1,\ldots,x_R) \sim \nu^R$. 
Consider the expected contribution to $F(f)_{sim}$ 
until the first 
time the simulation 
leaves $W$ (either in the decorrelation or in the parallel stage),
\begin{equation*}
\Delta F(f)_{sim} \triangleq \mathbb{E}^x\left[\int_0^{t_c \wedge T} f(X(s))\, ds\right] + 
\mathbb{E}^{x,\nu^R}\left[\mathbbm{1}_{T > t_c} \sum_{r=1}^R\int_0^{T^*} f(X^r(s))ds\right],
\end{equation*} 
where $\mathbb{E}^{x,\nu^R}$ 
denotes expectation for $(X(t), X^1(t), \ldots, X^R(t))$ with $X(t)$ starting at $x$ and the replicas $(X^1(t), \ldots,X^R(t))$ starting at initial distribution $\nu^R$. The error compared to direct (serial) simulation 
satisfies the bound 
\begin{equation}\label{CTMC_error_one_cycle}
\left|\mathbb{E}^x\left[\int_0^T f(X(s)) ds\right]-\Delta F(f)_{sim}
 \right|  \leq \|f\|_{\infty}\sup_{x \in W} \mathbb{E}^x\left[T\right] \|\nu_{t_c} - \nu\|_{\mathrm{TV}}.
\end{equation}
\end{theorem}

\begin{proof}
We estimate
\begin{equation*}
\begin{split}
&\left|\mathbb{E}^x\left[\int_0^T f(X(s)) ds\right]-\Delta F(f)_{\mathrm{sim}}\right|\\
=& \left|\mathbb{E}^x\left[\int_{t_c \wedge T}^T f(X(s)) ds\right] - {\mathbb E}^{x,\nu^R}\left[\mathbbm{1}_{T>t_c}\sum_{r=1}^R\int_0^{T^*} f(X^r(s))ds\right]\right|\\
=&\left|\left.\mathbb{E}^x\left[\int_{t_c}^{T} f(X(s)) ds \,\right|\,{T>t_c}\right] - {\mathbb E}^{x,\nu^R}\left.\left[\sum_{r=1}^R\int_0^{T^*} f(X^r(s))ds\right|\,{T > t_c}\right]\right| \mathbb{P}^x(T>t_c)\\
\le& \left|\left.\mathbb{E}^x\left[\int_{t_c}^{T} f(X(s)) ds \,\right|\,{T>t_c}\right] -{\mathbb E}^{\nu^R}\left[\sum_{r=1}^R\int_0^{T^*} f(X^r(s))ds\right]\right|,
\end{split}
\end{equation*}
where we used the fact that $X(t)$ and the replicas $(X^1(t), \ldots, X^R(t))$ are independent.
By the Markov property, 
\begin{equation*}
\left.\mathbb{E}^x\left[\int_{t_c}^{T} f(X(s)) ds \,\right|\,{T>t_c}\right] = \mathbb{E}^{\nu_{t_c}}\left[\int_0^T f(X(s)) ds\right].
\end{equation*}
By Theorem~\ref{thm:CTMC-parallel-unbiased},
\[ {\mathbb E}^{\nu^R}\left[\sum_{r=1}^R\int_0^{T^*} f(X^r(s))ds\right]
={\mathbb E}^\nu\left[\int_0^T f(X(s))\,ds\right].
\]
Combining the above estimates and equalities, 
\begin{align*}
&\left|\mathbb{E}^x\left[\int_0^T f(X(s)) ds\right]-\Delta F(f)_{sim}\right| \\
\le & \left|\mathbb{E}^{\nu_{t_c}}\left[\int_0^T f(X(s)) ds\right]-{\mathbb E}^\nu\left[\int_0^T f(X(s))\,ds\right]\right| \\
= & \left|\sum_{x\in W} \mathbb{E}^x\left[\int_0^T f(X(s)) ds\right] \nu_{t_c}(x) - \sum_{x \in W} \mathbb{E}^x\left[\int_0^{T} f(X(s))ds\right]\nu(x)\right|\\
\leq & \|f\|_{\infty}\sup_{x \in W} \mathbb{E}^x\left[T\right] \|\nu_{t_c} - \nu\|_{\mathrm{TV}}.
\end{align*}
\end{proof}

We  
note that $\mathbb{E}^x[T]$ is uniformly bounded in $x\in W$ if, for instance, 
$P_W$ is irreducible and 
$W$ is finite and non-absorbing for $X(t)$, as in Theorem~\ref{thm:QSD-spectral}. 
This uniform boundedness guarantees 
that the right hand side of~\eqref{CTMC_error_one_cycle} vanishes as $t_c \to \infty$.
%=============================================================

\section{The Embedded ParRep Method}

\subsection{Formulation of the Embedded ParRep Algorithm}
In this section, we introduce 
another algorithm for accelerating 
the computation of $\pi(f)$. 
The algorithm, called embedded ParRep, 
circumvents the disadvantages 
of CTMC ParRep discussed above.
As mentioned in the previous section,  CTMC ParRep can be slow due to the randomness of the holding times. 
In the worst case, one has to wait until all replicas leave $W$ in order to determine the first exit time $T^*$.
To circumvent this issue we propose an 
algorithm based on the embedded 
chain in which the parallel stage terminates as soon as one of the replicas leaves $W$.

Before we describe embedded ParRep, 
we introduce some notations. 
Throughout, $X_n^1,\ldots,X_n^R$ 
will be independent processes with 
the same law as $X_n$ and with initial 
distributions supported in~$W$. 
Moreover, we consider $X_n^1,\ldots,X_n^R$ as the embedded chains 
of $X^1(t),\ldots,X^r(t)$ defined 
above, and let
$\Delta \tau_n^1,\ldots,\Delta\tau_n^R$ be the corresponding holding 
times.
Recall that the first exit time of $X_n$ from $W$ is 
\[
N = \inf\{n > 0\,:X_n \notin W\}.
\]
For $r =1,\ldots,R$, we define the first exit time of $X_n^r$ from $W$ by
\[
N^r = \min\{n \in \mathbb{N}; X_n^r \notin W\} 
\]
and the smallest among them by 
\[
N^* = \min \{N^r; r=1,\ldots, R\}.
\] 
Note that it is possible that more than one replica leave $W$ for the 
first time after $N^*$ transitions. We denote by $K$ the smallest index among these escaped replicas. 
That is,
\[
K = \min\{r=1, \ldots, R ; X_{N^*}^r \notin W\}.
\]
It is clear from the above definition that $N^K = N^*$. Of course $N$, $N^r$, $N^*$ 
and $K$ depend on $W$, but we do not 
make this explicit.

Here and below we write ${\mathbb E}^{\mu^R}$ for expectation of 
$(X_n^1,\ldots,X_n^R)$ starting 
at $\mu^R$, where
\begin{equation*}
\mu^R(x_1,\ldots,x_R) = \prod_{r=1}^R \mu(x_r), \qquad x_1,\ldots,x_R \in W.
\end{equation*}
We begin by 
reproducing from~\cite{ParRep-Chain} Theorem~\ref{thm:geometric-dist} 
and~\ref{thm:same-dist} below, with proofs for completeness. 
\begin{theorem}\label{thm:geometric-dist}
Suppose $(X_n^1, \ldots, X_n^R)$ has 
initial distribution $\mu^R$.
Then $R(N^K-1) + K$ has the same distribution as $N^1$. 
\end{theorem}
\begin{proof}
%By Theorem~\ref{thm:QSD-property}, 
%$N_1$ is geometrically distributed with rate $\mathbb{P}^{\mu}(N^1>1)$.
Note that for any $n \ge 0$ and $k = 1, \ldots, R$, the event $\{N^K=n, K = k\}$ is equivalent to the event 
$\{N^1 > n, \ldots, N^{k-1} > n, N^k = n, N^{k+1} > n-1, \ldots, N^{R} > n-1\}$. Since $X_n^1,\ldots,X_n^R$ are iid and $N^1$ is geometrically distributed with rate $p=\mathbb{P}^{\mu^R}(N^1>1)$ (see Theorem~\ref{thm:QSD-property}), 
\[
\mathbb{P}^{\mu^R}(N^K = n, K = k) = (1-p)^{n(k-1)}(1-p)^{n-1}p (1-p)^{(n-1)(k-1)} = (1-p)^{R(n-1)+k-1}p.
\]
That is, $R(N^K-1) + K$ has geometric distribution with rate $p$. 
\end{proof}

\begin{theorem}\label{thm:same-dist}
Suppose $(X_n^1, \ldots, X_n^R)$ has 
the initial distribution $\mu^R$. Then
$X_{N^K}^K$ is independent of $R(N^K-1)+K$ and  
the distribution of $(X_{N^K}^{K}, R(N^K-1)+K)$ is same as that of $(X_{N^1}^1, N^1)$.
\end{theorem}
\begin{proof}
We first prove that $X_{N^K}^K$ is independent of $K$. Since $X_n^R,\ldots,X_n^R$ are iid and $N^k$ 
is independent of $X_{N^k}^k$ for each $k$, then $X_{N^k}^k$ is independent of $N^1, \ldots, N^R$.
Note that $K \in \sigma(N^1, \ldots, N^R)$, hence $X_{N^k}^k$ is independent of $K$ for each $k$.  
Now observe that for any $A \subset E$, 
\begin{equation*}
\begin{split}
\mathbb{P}^{\mu^R}(X_{N^K}^K \in A) &= \sum_{r=1}^R \mathbb{P}^{\mu^R}(X_{N^k}^k \in A, K = r)  \\
            & = \sum_{r=1}^R \mathbb{P}^{\mu^R}(X_{N^1}^1 \in A) \mathbb{P}^{\mu^R}(K = r)\\
            & = \mathbb{P}^{\mu^R}(X_{N^1}^1 \in A),
\end{split}
\end{equation*}
that is, $X_{N^K}^K$ and $X_{N^1}^1$ are equally distributed.  This implies that $X_{N^K}^K$ is independent of $K$.
To see the independence between $X_{N^K}^K$ and $R(N^K - 1) + K$, note that
\begin{equation*}
\begin{split}
\mathbb{P}^{\mu^R}(X_{N^K}^K \in A, N^K = n, K = r) &= \mathbb{P}^{\mu^R}(X_{N^r}^r \in A, N^r = n, K = r) \\
&= \mathbb{P}^{\mu^R}(X_{N^r}^r \in A, K = r | N^r = n) \mathbb{P}^{\mu^R}(N^r = n)\\
&= \mathbb{P}^{\mu^R}(X_{N^r}^r \in A | N^r = n) \mathbb{P}^{\mu^R}(N^r = n, K = r)\\
&= \mathbb{P}^{\mu^R}(X_{N^r}^r \in A) \mathbb{P}^{\mu^R}(N^r = n, K = r)\\
&= \mathbb{P}^{\mu^R}(X_{N^K}^K \in A) \mathbb{P}^{\mu^R}(N^K = n, K = r)
\end{split}
\end{equation*}
for any measurable $A \subset E$, $n \in \mathbb{Z}^+$ and $r = 1,\ldots, R$. Finally, Theorem~\ref{thm:geometric-dist} and the above analysis imply that 
$(X_{N^K}^{K}, R(N^K-1)+K)$ and $(X_{N^1}^1, N^1)$ are equally distributed. 
\end{proof}

Now we present the embedded ParRep algorithm in Algorithm~\ref{alg:embedded-ParRep}. 
In this algorithm we will need user-chosen parameters 
$n_c$ associated with 
each metastable set $W$. Roughly, these 
parameters correspond to the 
time for $X_n$ to converge 
to the QSD in $W$.

\begin{algorithm}[ht!]
\caption{Embedded ParRep}
\label{alg:embedded-ParRep}
\begin{algorithmic}[1]
\STATE

Set a decorrelation threshold $n_c$ for each metastable set $W$. Initialize the
 simulation time clock $N_{\mathrm{sim}}=0$ and the accumulated value $F(f)_{\mathrm{sim}}=0$. We will 
write $X^{\mathrm{par}}_n$ and $\Delta^{\mathrm{par}}_n$ 
for a DTMC and holding time process following the law of the embedded chain and holding times of $X(t)$ respectively. 
A complete ParRep cycle consists of three stages.
%---------------------------------------------------------------------------------------
\STATE
 {\bf Decorrelation Stage}:
   Starting at $n = N_{\mathrm{sim}}$, 
   evolve $X^{\mathrm{par}}_n$ and $\Delta \tau^{\mathrm{par}}_n$ until $X^{\mathrm{par}}_n$ spends $n_c$ consecutive time steps inside of the same metastable set $W$. That is, evolve 
   $X^{\mathrm{par}}_n$ and $\Delta \tau^{\mathrm{par}}_n$ from time $n = N_{\mathrm{sim}}$ 
   until time 
   \begin{equation*}
   N_{\mathrm{corr}} = \inf\{n\ge N_{\mathrm{sim}}+n_c-1: X_m^{\mathrm{par}} \in W ~\mathrm{for} ~m \in \{n-n_c+1,\ldots,n\}~\mathrm{for~some}~W\}.
   \end{equation*}
   Then update 
   \begin{equation*}
   F(f)_{\mathrm{sim}} = F(f)_{\mathrm{sim}}+ \sum_{n=N_{\mathrm{sim}}}^{N_{\mathrm{corr}}-1} f(X^{\mathrm{par}}_n)\Delta \tau^{\mathrm{par}}_n,
   \end{equation*}
   set $N_{\mathrm{sim}} = N_{\mathrm{corr}}$, and 
   proceed to the dephasing stage.
%---------------------------------------------------------------------------------------
\STATE
 {\bf Dephasing Stage} : Let $W$ be 
such that $X^{\mathrm{par}}_{N_{\mathrm{sim}}} \in W$, 
that is, $W$ is the metastable set from the end of the decorrelation stage. Generate $R$ independent samples $x_1, \ldots, x_R$ from $\mu$, the QSD of $X_n$ in $W$. Then proceed to the parallel stage.  
%---------------------------------------------------------------------------------------
\STATE
{\bf Parallel Stage} : Start $R$ parallel processes $X^1_n, \ldots, X^R_n$ at $ x_1, \ldots, x_R$, and evolve them and 
the corresponding holding times 
$\Delta \tau^1_n,\ldots,\Delta \tau^R_n$ from time $n=0$ until time $N^*$. Then update
     \begin{equation}\label{eqn:embedded-parallel-stage}
     \begin{split}
    & F(f)_{\mathrm{sim}} = F(f)_{\mathrm{sim}} + \sum_{r=1}^{R} \sum_{k=0}^{N^*-2}f(X_k^r) \Delta \tau_k^r + \sum_{r=1}^{K} f(X_{N^*-1}^r) \Delta\tau_{N^*-1}^r\\
     &N_{\mathrm{sim}}= N_{\mathrm{sim}}+R(N^*-1)+K,
     \end{split}
     \end{equation}
    set ${X}_{N_{\mathrm{sim}}}^{\mathrm{par}} = X_{N^*}^{K}$, and return to the decorrelation stage.
%------------------------------------------------------------------------------------------------------
\STATE 
     The algorithm is stopped when $N_{\mathrm{sim}}$ reaches some 
     user-chosen time $N_{\mathrm{end}}$. The stationary average $\pi(f)$ is estimated as 
     \[
     \pi(f)\approx F(f)_{\mathrm{sim}}/F(1)_{\mathrm{sim}}.
     \] 
\end{algorithmic}
\end{algorithm}

\begin{figure}[ht]
\centering
\includegraphics[scale=0.75]{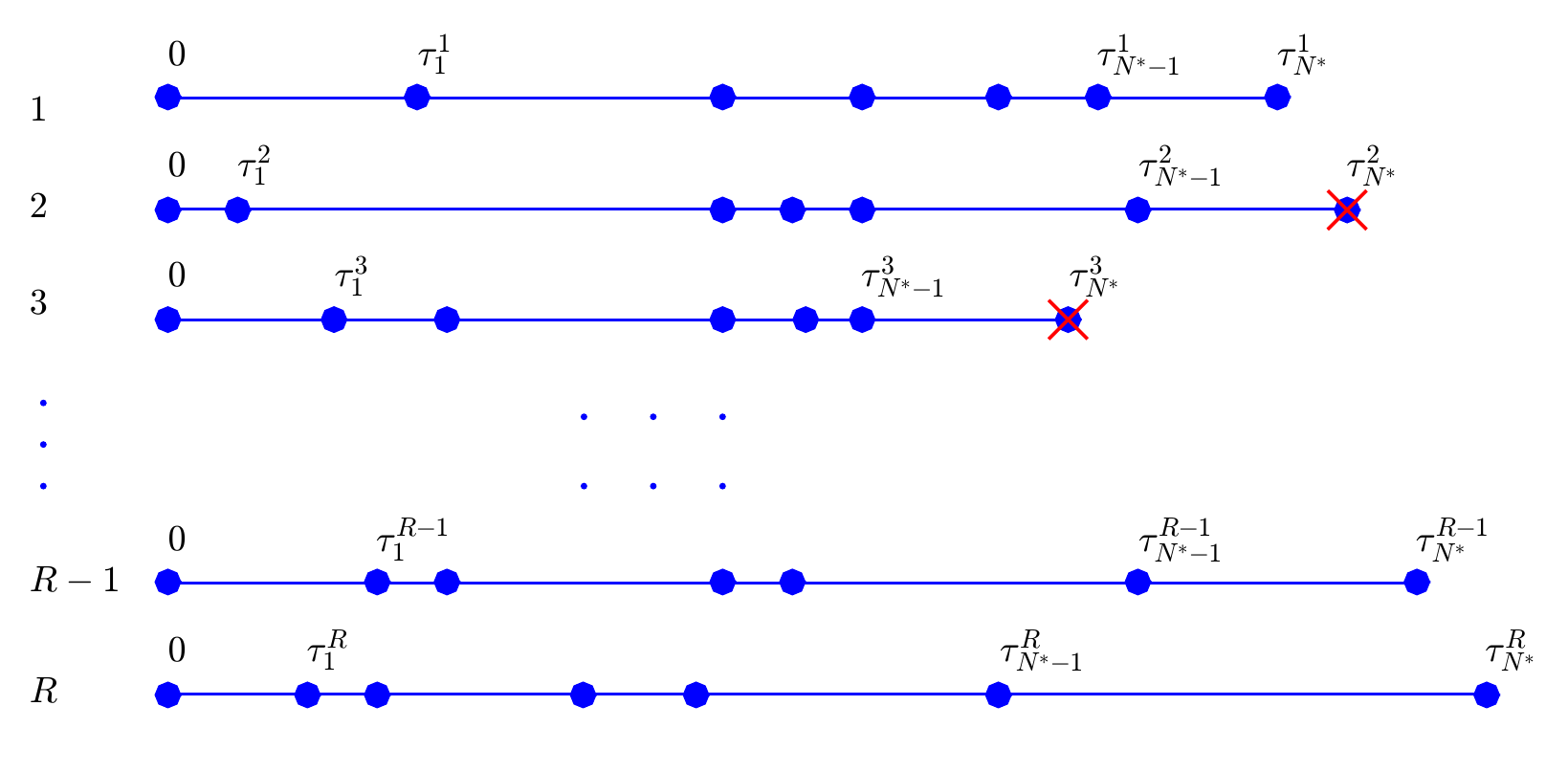}
\caption{The diagram for one parallel stage of the embedded ParRep algorithm with $R$ replicas. Each blue dot represents an exit event along the time line. Both replica $2$ and $3$ leave $W$ after $N^* = 6$ transitions (the blue dot with the red ``x''), in which case $K = 2$.}
\label{fig:embedded_ParRep}
\end{figure}

The DTMC $X_n$ and holding times 
$\Delta \tau_n$ are
simulated by the stochastic simulation algorithm (SSA), see, for instance, \cite{gillespie1977exact}, just as in the CTMC ParRep.
If $X^{\mathrm{par}}_n$ remains in $W$ for sufficiently long time (i.e., time $t_c$), it is distributed nearly 
according to the QSD $\mu$ of $X_n$ in $W$. See Theorem~\ref{thm:QSD-spectral}. This means 
that at the end of the decorrelation 
stage, $X^{\mathrm{par}}_n$ can 
be considered a sample of $\mu$.

The aim of the dephasing stage is to prepare a sequence of iid initial states with distribution $\mu$. 
Like the CTMC ParRep, rejection sampling can be used for the embedded ParRep as well. 
However, a more natural and efficient option for the embedded ParRep
is a Fleming-Viot 
based sampling procedure~\cite{binder2015generalized,ferrari2007quasi}. The 
procedure can be summarized as follows. 

The $R$ replicas 
$X^1_n,\ldots,X^R_n$, starting in $W$, evolve until one or more of them leaves $W$. 
Then each replica that left $W$ is restarted from the current state 
of another replica that is currently in $W$ (usually chosen uniformly at random). 
The procedure stops after the replicas have evolved for $n = n_p$ 
time steps, where $n_p$ is a parameter similar to $n_c$. (If all the 
replicas leave $W$ at the same time, the procedure restarts from 
the beginning.) With this type of sampling, the number of time 
steps simulated for each replica in the dephasing step is the same. 
In particular, if the $R$ parallel processors are synchronous (i.e. 
if each processor takes the same wall clock time to simulate one time step), then 
each processor finishes the dephasing step at the same wall clock time.
We comment that the Fleming Viot technique can be used to estimate the decorrelation and 
dephasing thresholds as well when they are difficult to choose a priori \cite{binder2015generalized}. 

The acceleration of the embedded ParRep comes from the parallel stage. 
Roughly, we only have to wait $N^*$ time steps instead of $N$ to observe 
an exit from $W$. The theoretical wall clock time speedup can 
be approximately a factor of $R$. See Theorem~\ref{thm:geometric-dist} 
below. 
Like with CTMC ParRep, the parallel step does not 
require metastability for this time speedup, but if 
$W$ is not metastable, then the dephasing step will not be efficient. 
See the remarks below Algorithm~\ref{alg:CTMC-ParRep}.

Similar to the CTMC ParRep, each parallel stage of the embedded ParRep has a consistent averaged contribution to the stationary average.   
Suppose that $x_1,\ldots,x_R$ are iid samples from $\mu$.
\begin{itemize}
\item [1.] The joint law of $(X_{N^K}^K,R(N^K-1)+K)$ is the same as that of $(X_{N^1}^1,N^1)$. That is, the joint distribution of the first exit time and the exit state for each parallel stage is independent of the number of replicas.
\item [2.] The expected value of 
\[
\sum_{r=1}^{R} \sum_{k=0}^{N^*-2}f(X_k^r) \Delta \tau_k^r + \sum_{r=1}^{K} f(X_{N^*-1}^r) \Delta\tau_{N^*-1}^r
\] is the same as that of 
\[
\sum_{n=0}^{N-1} f(X_n^1)\Delta \tau_n^1.
\]
Hence the expected contribution to $F(f)_{\mathrm{sim}}$ from each parallel stage is independent of the number of replicas. See Theorem~\ref{thm:embedded-parallel-unbiased} below.
\end{itemize}
See Theorem~\ref{thm:same-dist} 
and~\ref{thm:embedded-parallel-unbiased} 
for proofs of these statements.

%Consistency of stationary average follows from 1-2 above and 
%the law of large numbers: since there are indefinitely many parallel stages in a given $W$, consistency is ensured so long as the
%expected contribution to $F(f)_{sim}$ 
%from the parallel stage has 
%the correct expected value. 

We expect that embedded ParRep is superior to the CTMC ParRep for the following two reasons. 
First, consider the parallel 
stages of both algorithms. 
In the CTMC ParRep, observing the first exit event in the parallel stage is not sufficient to determine $T^*$.
But in embedded ParRep, once any replica leaves $W$, we know $N^*$. Thus the embedded ParRep parallel step terminates 
once any of the replicas leaves $W$. 
For this reason we expect the parallel 
stage of the embedded ParRep to be 
significantly
faster than that of the CTMC ParRep.
Second, consider the dephasing stage.
For the embedded ParRep,  
Fleming-Viot sampling is a natural  
technique because if the processors 
are synchronous then they all 
finish the dephasing stage at the same 
wall clock time, and only the current 
states of each processor are 
needed at each time step to decide where to restart replicas 
which left $W$. 
For asynchronous processors, one can simply implement 
a polling time.
This is not true, however, for Fleming-Viot 
sampling with the CTMC ParRep. Indeed, 
to implement Fleming-Viot sampling with the 
CTMC ParRep, one would have to store 
the histories of every replica, and 
the replicas would finish at potentially 
very different wall clock times.
The rejection method can be slow 
for both algorithms, 
particularly when the metastability is weak  or when the 
number of replicas is large.

%Our ultimate goal is to estimate the stationary average of any function $f$. Therefore, eventually we compute the quantity $F(f)_{\mathrm{sim}}/F(1)_{\mathrm{sim}}$ in order to approximate $\pi(f)$, where $\pi$ is the stationary distribution of the underlying CTMC $X(t)$.

%=====================================================
\subsection{Error analysis of the embedded ParRep}

Now we are able to show that if the dephasing sampling is exact, 
then the expected contribution to $F(f)_{\mathrm{sim}}$ from the parallel stage is exact.

\begin{theorem}\label{thm:embedded-parallel-unbiased}
Suppose in the dephasing step $(x_1,\ldots,x_R) \sim \mu^R$. Then 
the expected contribution to $F(f)_{sim}$ 
from the parallel stage of 
Algorithm~\ref{alg:embedded-ParRep} 
is the same for every number of replicas.

\begin{equation*}
\begin{split}
\mathbb{E}^{\mu^R}\left[\sum_{r=1}^{R} \sum_{k=0}^{N^*-2}f(X_k^r) \Delta \tau_k^r + \sum_{r=1}^{K} f(X_{N^*-1}^r) \Delta\tau_{N^*-1}^r\right] &= {\mathbb E}^{\mu}\left[\sum_{n=0}^{N-1}f(X_n)\Delta \tau_n\right]
= \mu(f q^{-1})\mathbb{E}^{\mu}\left[N\right],
\end{split}
\end{equation*}
where $q$ is the function as defined in section \ref{sec:assume}.
\end{theorem}
\begin{proof}
We first rewrite 
\begin{align}\begin{split}\label{eqn:time-multiple2}
&\sum_{r=1}^{R} \sum_{k=0}^{N^*-2}f(X_k^r) \Delta \tau_k^r + \sum_{r=1}^{K} f(X_{N^*-1}^r) \Delta\tau_{N^*-1}^r \\
&= \sum_{r=1}^{R}\sum_{i=0}^{N^*-1}f(X_i^r)\Delta\tau_i^r - \sum_{r=K+1}^{R}f(X_{N^*-1}^r)\Delta \tau_{N^*-1}^r.
\end{split}
\end{align} 
For the first part, we condition $N^*$ and obtain
\[
\mathbb{E}^{\mu^R}\left[\sum_{r=1}^{R}\sum_{i=0}^{N^*-1}f(X_i^r) \Delta\tau_i^r\right] = \sum_{r=1}^{R} \sum_{n=1}^{\infty} \sum_{i=0}^{n-1} \mathbb{E}^{\mu^R}\left[f(X_i^r) \Delta\tau_i^r \mathbbm{I}_{N^* = n} \right]
\]
Interchanging the iterated summations leads to 
\[
\sum_{r=1}^{R}\sum_{n=1}^{\infty} \sum_{i=0}^{n-1} \mathbb{E}^{\mu^R}\left[f(X_i^r) \Delta\tau_i^r \mathbbm{I}_{N^* = n} \right] = \sum_{r=1}^{R} \sum_{i=0}^{\infty}\mathbb{E}^{\mu^R}\left[f(X_i^r) \mathbbm{I}_{N^*>i}\Delta \tau_i^r \right].
\]
Notice $N^* > i$ is equivalent to $N^1 > i, \ldots, N^R > i$ and $\Delta \tau_i^r$ is independent of $N^s$ for $s \ne r$. Thus
\begin{equation*}
\begin{split}
&\sum_{r=1}^{R}\sum_{i=0}^{\infty}\mathbb{E}^{\mu^R}\left[f(X_i^r)\Delta\tau_i^r | N^*>i\right]\mathbb{P}^{\mu^R}({N^*>i})\\
 =& \sum_{r=1}^{R} \sum_{i=0}^{\infty} \mathbb{E}^{\mu^R}\left[f(X_i^r) \Delta\tau_i^r | N^r > i\right] \mathbb{P}^{\mu^R}(N^* > i).
\end{split}
\end{equation*}
Now by Lemma~\ref{lem:conditioned-exp} 
and the definition of the QSD,
\begin{align*}
\mathbb{E}^{\mu}\left[f(X_i^r) \Delta\tau_i^r | N^r > i\right] &= \mathbb{E}^{\mu}\left[\left.{\mathbb E}^{\mu}\left[\left. f(X_i^r) \Delta\tau_i^r \right| \{X_n^r\}_{n=0,1,\ldots}\right]\right| N^r > i\right]  \\
&= \mathbb{E}^{\mu}\left[\left.f(X_i^r){\mathbb E}^{\mu}\left[ \left.\Delta\tau_i^r \right| \{X_n^r\}_{n=0,1,\ldots}\right]\right| N^r > i\right]  \\
&= \mathbb{E}^{\mu}\left.\left[f(X_i^r)q(X_i^r)^{-1}\right| N^r > i\right] = \mu(fq^{-1}).
\end{align*}
Combining the last four equations gives
\begin{equation}\label{eqn:part1}
\mathbb{E}^{\mu^R}\left[\sum_{r=1}^{R}\sum_{i=0}^{N^*-1}f(X_i^r) \Delta\tau_i^r\right] = \mu(f q^{-1}) R\mathbb{E}^{\mu^R}[N^*].
\end{equation}

A similar argument can be applied to the second term on the right hand 
side of \eqref{eqn:time-multiple2}. First we condition $N^*$ and $K$ simultaneously such that
\begin{equation*}
\begin{split}
&\mathbb{E}^{\mu^R}\left[\sum_{r=K+1}^{R}f(X_{N^*-1}^r) \Delta \tau_{N^*-1}^r\right] \\
=&\sum_{n=1}^{\infty}\sum_{r=1}^{R}\sum_{r=k+1}^{R} \mathbb{E}^{\mu^R}\left[f(X_{n-1}^r)\Delta \tau_{n-1}^r | N^*=n, K=k\right]\mathbb{P}^{\mu^R}(N^*=n, K=k).
\end{split}
\end{equation*}
Interchanging the second and third summations the right-hand side equals 
\[
\sum_{n=1}^{\infty}\sum_{r=2}^{R}\sum_{k=1}^{r-1} \mathbb{E}^{\mu^R}\left[f(X_{n-1}^r)\Delta \tau_{n-1}^r | N^*=n, K=k\right]\mathbb{P}^{\mu^R}(N^*=n, K=k)
\] 
Recall that 
\begin{equation*}
N^* = n,\, K=k \Longleftrightarrow N^1>n, \ldots, N^{k-1}>n,\, N^k = n, \,N^{k+1}>n-1, \ldots, N^{R} > n-1.
\end{equation*}
 Thus, using independence of $X_n^1,\ldots,X_n^R$ and the definition of the QSD,
\begin{equation*}
\begin{split}
&\sum_{n=1}^{\infty}\sum_{r=2}^{R}\sum_{k=1}^{r-1} \mathbb{E}^{\mu^R}\left[f(X_{n-1}^r)\Delta \tau_{n-1}^r | N^*=n, K=k\right]\mathbb{P}^{\mu^R}(N^*=n, K=k)\\
=&\sum_{n=1}^{\infty}\sum_{r=2}^{R}\sum_{k=1}^{r-1} \mathbb{E}^{\mu}\left[f(X_{n-1}^r)\Delta \tau_{n-1}^r | N^r > n-1\right]\mathbb{P}^{\mu^R}(N^*=n, K=k)\\
= &\mu(f q^{-1})\sum_{n=1}^{\infty}\sum_{r=2}^{R}\sum_{k=1}^{r-1} \mathbb{P}^{\mu^R}(N^*=n, K=k)\\
= &\mu(f q^{-1})(R-\mathbb{E}^{\mu^R}[K]).
\end{split}
\end{equation*}
Combining the last three equations leads to 
\begin{equation}\label{eqn:part2}
\mathbb{E}^{\mu}\left[\sum_{r=K+1}^{R}f(X_{N^*-1}^r) \Delta \tau_{N^*-1}^r\right] = \mu(f q^{-1})(R-\mathbb{E}^{\mu^R}[K]).
\end{equation}
Subtracting \eqref{eqn:part2} from \eqref{eqn:part1}, we have
\[
{\mathbb E}^{\mu^R}\left[\sum_{r=1}^{R}\sum_{i=0}^{N^*-1}f(X_i^r)\Delta\tau_i^r - \sum_{r=K+1}^{R}f(X_{N^*-1}^r)\Delta \tau_{N^*-1}^r\right]
= \mu(f q^{-1})\mathbb{E}^{\mu^R}\left[R(N^*-1)+K\right].
\]
Now the result follows since 
\[
\mu(f q^{-1})\mathbb{E}^{\mu^R}\left[R(N^*-1)+K\right]= \mu(f q^{-1})\mathbb{E}^{\mu}[N]
\]
by Theorem~\ref{thm:same-dist}.
In particular, when $R=1$ we 
have $N^*=N$ and $K=1$, and thus 
\[
\mathbb{E}^{\mu}\left[\sum_{n=0}^{N-1} f(X_n) \Delta \tau_n\right] =  \mu(fq^{-1}) \mathbb{E}^{\mu}[N].
\]
\end{proof}
%We remark that the result holds under the %assumption that all replicas are %independent of each other which is true if %the rejection sampling is used for the %dephasing stage. 
%However, one can only expect a sequence of %weakly dependent random variables   
%from the dephasing stage if the Fleming-%Viot particle technique is utilized. 

%-----------------------------------------

We now 
prove an analog of Theorem~\ref{thm:CTMC-F} for the embedded ParRep. 
Recall we have assumed 
convergence of $\|\mu_{n_c}-\mu\|_{\mathrm{TV}} \to 0$ as $n_c \to \infty$, for 
every starting point $x \in E$.
%where $\|\cdot\|$ denotes total variation 
%norm. 
See for instance~Theorem~\ref{thm:QSD-spectral} for conditions guaranteeing 
this convergence.

\begin{theorem}\label{thm:embedded-F}
Consider the embedded ParRep 
starting at $ x \in W$ in the decorrelation stage.
Assume the dephasing stage sampling is exact, that is, $(x_1,\ldots,x_R) \sim \mu^R$.
Consider the expected contribution to $F(f)_{sim}$ 
up until the first 
time the simulation 
leaves $W$ (either in the 
decorrelation stage or in the parallel stage):
\begin{equation*}
\begin{split}
\Delta F(f)_{\mathrm{sim}} \triangleq \mathbb{E}^x  &\left[ \sum_{n=0}^{n_c \wedge N -1} f(X_n)\Delta \tau_n\right] 
+  \mathbb{E}^{x, \mu^R}\left[ \mathbbm{1}_{N>n_c} \sum_{r=1}^{R} \sum_{k=0}^{N^*-2}f(X_k^r) \Delta \tau_k^r \right.\\
&+\mathbbm{1}_{N>n_c} \left. \sum_{r=1}^{K} f(X_{N^*-1}^r) \Delta\tau_{N^*-1}^r\right],
\end{split}
\end{equation*} 
where $\mathbb{E}^{x,\mu^R}$ 
denotes expectation for $(X_n, X_n^1, \ldots, X_n^R)$ with $X_n$ starting at $x$ and the replicas $(X^1_n, \ldots,X^R_n)$ starting 
at the initial distribution $\mu^R$. 
The error compared to a direct (serial) simulation satisfies the bound 
\begin{equation}\label{embedded_error_one_cycle}
\left|\mathbb{E}^x\left[\sum_{n=0}^{N-1} f(X_n)\Delta \tau_n \right]-\Delta F(f)_{sim}
 \right|  \leq \|f\|_{\infty}\sup_{x \in W} \mathbb{E}^x\left[T\right] \|\mu_{n_c} - \mu\|_{\mathrm{TV}}.
\end{equation}
\end{theorem}

\begin{proof}

The proof is similar to that for the CTMC ParRep,  
\begin{equation*}
\begin{split}
&\left|\mathbb{E}^x\left[\sum_{n=0}^{N-1} f(X_n)\Delta \tau_n \right]-\Delta F(f)_{\mathrm{sim}}\right|\\
= &\left| \mathbb{E}^x   \left[\sum_{n=n_c\wedge N}^{N-1} f(X_n)\Delta \tau_n\right]
 - \mathbb{E}^{x,\mu^R}\left[\mathbbm{1}_{N>n_c}\sum_{r=1}^{R} \sum_{k=0}^{N^*-2}f(X_k^r) \Delta \tau_k^r \right. \right. \\
  &+ \mathbbm{1}_{N>n_c}\left.\left.\sum_{r=1}^{K} f(X_{N^*-1}^r) \Delta\tau_{N^*-1}^r  \right]  \right|\\
\leq & \left|\mathbb{E}^x\left.\left[\sum_{n=n_c}^{N-1} f(X_n)\Delta \tau_n\right| N>n_c \right]\right.\\
 &- \left.\mathbb{E}^{\mu^R}\left[\sum_{r=1}^{R} \sum_{k=0}^{N^*-2}f(X_k^r) \Delta \tau_k^r + \sum_{r=1}^{K} f(X_{N^*-1}^r) \Delta\tau_{N^*-1}^r\right]\right|.\\ 
\end{split}
\end{equation*}
%=& \left|    \mathbb{E}^x   \left[\sum_{n=n_c}^{N-1} f(X_n)\Delta \tau_n\right| N>n_c \right]- \mathbb{E}^{x,\mu^R}\left[\sum_{r=1}^{R} \sum_{k=0}^{N^*-2}f(X_k^r) \Delta \tau_k^r \right.\\
%&+\left.\left. \left. \sum_{r=1}^{K} f(X_{N^*-1}^r) \Delta\tau_{N^*-1}^r \right| N>n_c\right]\right| {\mathbb P}^\mu(N>n_c)\\
By the Markov property
\begin{equation*}
\mathbb{E}^x\left.\left[\sum_{n=n_c}^{N-1} f(X_n)\Delta \tau_n\right| N>n_c \right] = \mathbb{E}^{\mu_{n_c}}\left[\sum_{n=0}^{N-1} f(X_n)\Delta \tau_n\right].
\end{equation*}
Owing to Theorem~\ref{thm:embedded-parallel-unbiased},
\[ \mathbb{E}^{\mu^R}\left[\sum_{r=1}^{R} \sum_{k=0}^{N^*-2}f(X_k^r) \Delta \tau_k^r + \sum_{r=1}^{K} f(X_{N^*-1}^r) \Delta\tau_{N^*-1}^r\right] =
{\mathbb E}^{\mu}\left[\sum_{n=0}^{N-1}f(X_n)\Delta \tau_n\right].
\]
Therefore
\begin{align*}
&\left|\mathbb{E}^x\left[\sum_{n=0}^{N-1} f(X_n)\Delta \tau_n \right]-\Delta F(f)_{\mathrm{sim}}\right| \\
&\le 
\left|\mathbb{E}^{\mu_{n_c}}\left[\sum_{n=0}^{N-1} f(X_n)\Delta \tau_n\right]
-
{\mathbb E}^{\mu}\left[\sum_{n=0}^{N-1}f(X_n)\Delta \tau_n\right]
\right| \\
&= \left|\sum_{x \in W} 
\mathbb{E}^{x}\left[\sum_{n=0}^{N-1} f(X_n)\Delta \tau_n\right]\mu_{n_c}(x) - 
\sum_{x \in W} {\mathbb E}^{x}\left[\sum_{n=0}^{N-1}f(X_n)\Delta \tau_n\right]\mu(x)\right| \\
&\le \|f\|_\infty \sup_{x \in W}{\mathbb E}^x[T] \|\mu_{n_c}-\mu\|_{TV}
\end{align*}
with the last equation coming from 
the fact that ${\mathbb E}^x[\sum_{n=0}^{N-1}\Delta \tau_n] = {\mathbb E}^x[T]$.
\end{proof}

%--------------------------------

%==================================
\section{Numerical Experiments}
We present two numerical examples from the stochastic reaction networks in order to demonstrate the consistency and efficiency of the ParRep algorithms. 
\subsection{Reaction networks with linear propensity}
We consider the following stochastic reaction network
\begin{equation}\label{eqn:reaction}
\varnothing \longrightarrow A \rightleftharpoons B \longrightarrow C \longrightarrow \varnothing
\end{equation}
taken from \cite{cotter2011constrained},
where $A, B$ and $C$ represent reacting species.
%We consider the case with initial population $A=1, B=1$ and $C=0$. 
The time evolution of the population (the number of species)  in the reaction network is commonly modeled as a CTMC 
$X(t) = (X_1(t), X_2(t), X_3(t))$ with state space $E \subset \mathbb{Z}_+^3$. 
%It is easy to see that $X(t)$ is irreducible and positive recurrent. 
%By convention,  we use the vector $x_0$ to denote the initial population. 
The jump rate of each reaction is governed by the propensity function (intensity) $\lambda_j(x), j = 1,\ldots, 5$ such that for all $t>0$,
\[
\lambda_j(x) = \lim_{h \to 0}\dfrac{\mathbb{P}(X(t+h) = x + \eta_j | X(t) = x)}{h},
\]
where $\eta_j$ is the state change vector associated with the $j$th reaction.  
We list the reactions and their corresponding propensity functions and state change vectors in Table~\ref{tab:reaction-network-linear}. 

\begin{table}[ht]
\centering
\caption{Reactions, propensity functions and state change vectors}
\begin{tabular}{l|l|l}
\hline
Reaction & Propensity function & State change vector\\
\hline\hline
$\varnothing \longrightarrow A $ & $\lambda_1(x) = c_1$ & $\eta_1=(1,0,0)$\\
\hline
$A \longrightarrow B$  & $\lambda_2(x) = c_2 x_1$  & $\eta_2 = (-1,1,0)$ \\
\hline
$B \longrightarrow A$  & $\lambda_3(x) = c_3 x_2$  & $\eta_3 = (1,-1,0)$\\
\hline
$B \longrightarrow C$  & $\lambda_4(x) = c_4 x_2$  & $\eta_4 = (0,-1,1)$\\
\hline
$C \longrightarrow \varnothing$  & $\lambda_5(x) = c_5 x_3$  & $\eta_5 = (0,0,-1)$\\
\hline
\end{tabular}
\label{tab:reaction-network-linear}
\end{table}

In this numerical experiment, we take the initial state $x_0=(5,10,10)$
and the rate constants 
\[
(c_1, c_2, c_3, c_4, c_5)=(0.1, 100, 100, 0.01, 0.01).
\]
With this choice of parameters the timescale separation is about $\epsilon = 10^{-4}$ and hence the process $X(t)$ demonstrates metastability. 
The reactions $A \to B$ and $B\to A$ occur with a much higher probability than the other reactions and hence
we call $A \to B$ and $B\to A$ fast reactions and the other reactions slow reactions. 
The occurrence of slow reactions is a rare event.
We define the observables $f_1(x) = x_1 + x_2$ and $f_2(x) = x_3$, the collection of sets $\{W_{m,n}\}_{m,n \in \mathbb{Z}_+}$ with 
\[
W_{m,n} = \{ x \in E : f_1(x) = m, f_2(x) = n\}
\]
form a full decomposition of the state space $E$. 
Note that both the total population of species $A$ and $B$ (i.e., $f_1(X(t))$) and the population of species $C$ (i.e. $f_2(X(t))$) 
remain constant until one of the slow reactions occurs. 
Hence the typical sojourn time for $X(t)$ in each $W_{m, n}$ is very long comparing to the transition time between 
any two states that are in $W_{m, n}$. 
In this case, we say $X(t)$ is metastable in $W_{m, n}$.
For example, with the initial population $x_0=(1,1,0)$, the states 
$(1,1,0), (2, 0, 0)$ and $(0,2,0)$ form a metastable set since the fast reactions $A \to B$ and $B \to A$ occur with a significantly 
higher probability than slow reactions and only the occurrence of the slow reactions can allow the process to move from the metastable set to another metastable set. 
Note that both observables $f_1$ and $f_2$ defined above are invariant in each metastable set, we call them slow observables.
In general, an observable $f$ is called a slow observable if it is invariant in each metastable set $W_{m, n}$, i.e., 
there is a constant $C(m, n)$ such that $f(x) = C(m, n)$ for each $x \in W_{m, n}$. 
An observable is called a fast observable if it is not slow (e.g., $f(x) = x_1$). 

This kind of two-scale problems arise in many fields other than the stochastic reaction networks, such as the queuing theory and population dynamics.
Estimation of the distributions of two-scale processes can be computationally prohibitive due to the insufficient sampling of the rare events.
Therefore, it is desirable to apply the two ParRep algorithms proposed in this paper to accelerate the long time simulation and estimate the stationary distribution. 

\begin{figure}[ht]
\centering 
\includegraphics[scale=0.6]{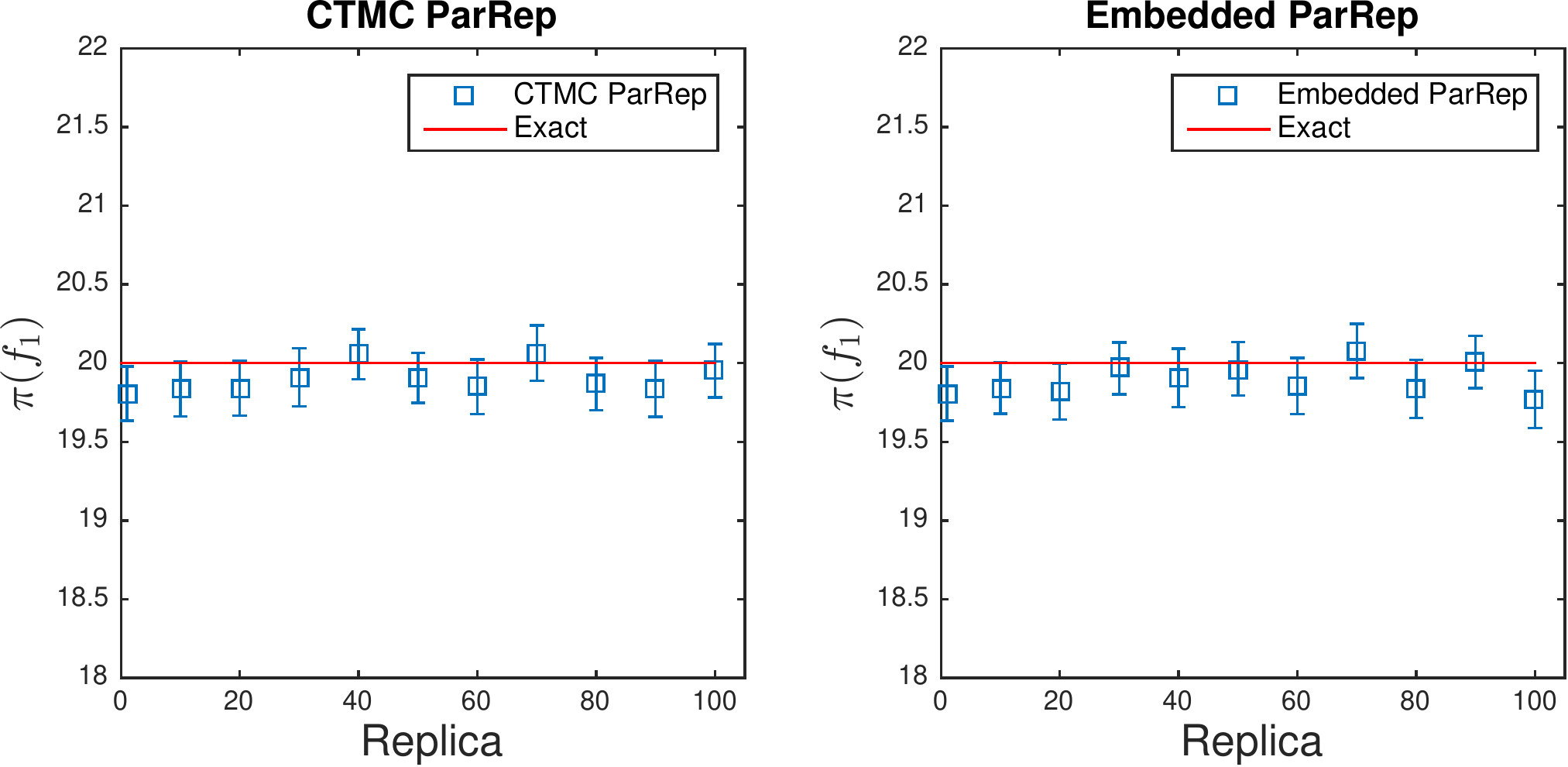}
\caption{The stationary average of the slow observable $f_1(x)=x_1+x_2$ computed with the CTMC ParRep (left) and with the embedded ParRep (right). The user-specified terminal time is $T_{\mathrm{end}} = 10^4$ in the simulation.}
\label{fig:slow_f1}
\end{figure}

\begin{figure}[ht!]
\centering 
\includegraphics[scale=0.6]{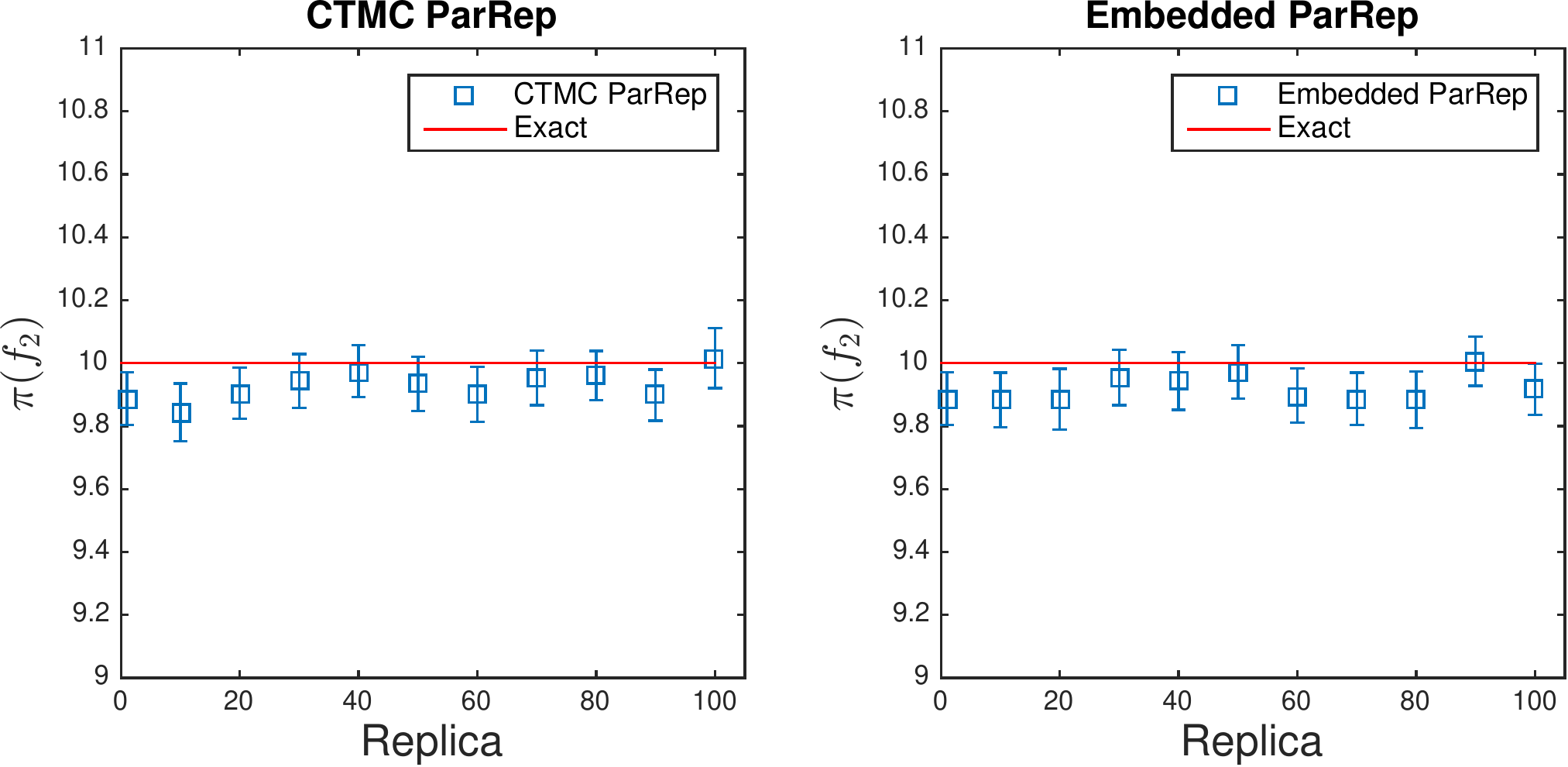}
\caption{The stationary average of the slow observable $f_2(x)=x_3$ computed with the CTMC ParRep (left) and with the embedded ParRep (right). The user-specified terminal time is $T_{\mathrm{end}} = 10^4$ in the simulation.}
\label{fig:slow_f2}
\end{figure}

\begin{figure}[ht!]
\centering 
\includegraphics[scale=0.6]{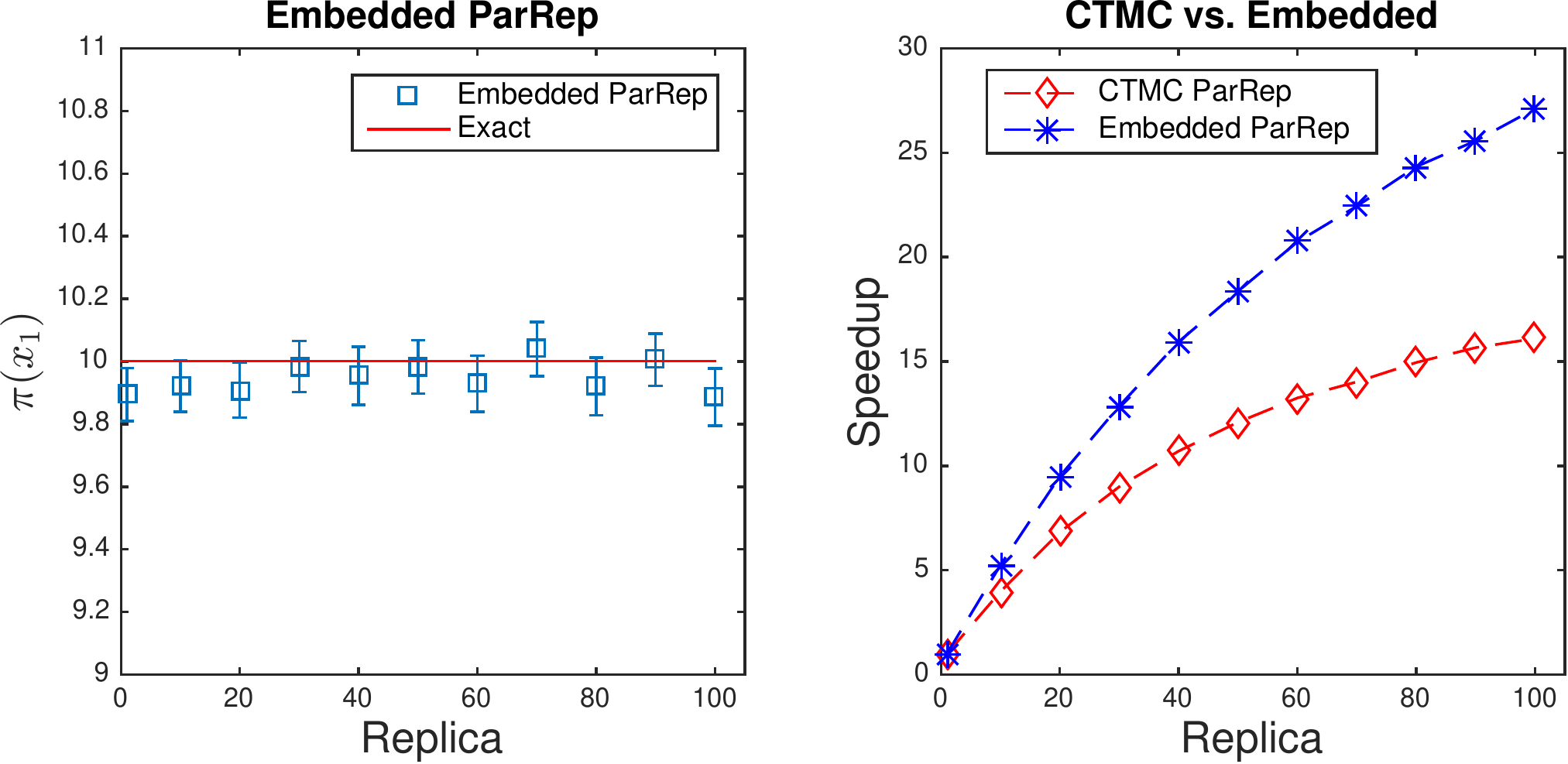}
\caption{The stationary average of the fast observable $f_3(x)=x_1$ computed with the embedded ParRep (left) and the speedup comparison between the CTMC ParRep and the embedded ParRep (right). The user-specified terminal time is $T_{\mathrm{end}} = 10^4$ in the simulation.}
\label{fig:fast_and_time}
\end{figure}

We apply both the CTMC ParRep and the embedded ParRep to estimate the stationary averages of the slow observables $f_1$ and $f_2$. 
The stationary distribution of the fast observable $f_3(x) = x_1$ is also computed using the embedded ParRep.
On the other hand, for the reaction network \eqref{eqn:reaction} under consideration, one can calculate the stationary distribution analytically since it only involves mono-molecular reactions.
In fact, it can be shown that the stationary distribution is a multivariate Poisson distribution \cite{cotter2011constrained}, that is,
\begin{equation}\label{eqn:exact}
\pi(x_1, x_2, x_3) = \frac{\bar{\lambda}_1^{x_1}\bar{\lambda}_2^{x_2}\bar{\lambda}_3^{x_3}}{x_1!x_2!x_3!} e^{-(\bar{\lambda}_1 + \bar{\lambda}_2+\bar{\lambda}_3)},
\end{equation} 
where 
\[
\bar{\lambda}_1 = \dfrac{c_1 (c_3 + c_4) }{c_2c_4},\quad \bar{\lambda}_2 = \dfrac{c_1}{c_4},  \quad   \bar{\lambda}_3 = \dfrac{c_1}{c_5}.
\]
Hence the exact stationary averages of the slow observables $f_1$ and $f_2$ are $\pi(f_1) = 20.001$ and $\pi(f_2) = 10$ and the exact stationary averages of the fast observable $f_3(x) = x_1$ is $10.001$.
We use this exact result to compare with our result from numerical simulation.

Our simulations compare the CTMC ParRep and the embedded ParRep with the Stochastic Simulation Algorithm (SSA), \cite{gillespie1977exact}.
In Figure~\ref{fig:slow_f1}, we demonstrate the estimation of $\pi(f_1)$ using the CTMC ParRep and the embedded ParRep with 
various numbers of replicas ($R=10,20,\cdots, 100$) and with SSA ($R=1$).
Similarly, Figure~\ref{fig:slow_f2} shows the estimation of $\pi(f_2)$.
Note that only the embedded ParRep is used to compute the stationary average of the fast variable $f(x) = x_1$ since the CTMC ParRep is not efficient for fast observables as we commented at the end of Section~\ref{sec:CTMC}.
Currently, the rejection sampling is used for dephasing and the decorrelation and dephasing thresholds 
are taken to be $t_c = t_p = 0.01$ for the CTMC ParRep and $n_c=n_p=15$ steps for the embedded ParRep.
In Figure~\ref{fig:fast_and_time}, the estimation for the fast observable and speedup 
are shown. 
It can be seen that with $10$ replicas,
the speedup factor is about $4.5$ for the CTMC ParRep and $5.5$ for
the embedded ParRep. 
When the number of replicas increases, the embedded ParRep becomes much more efficient than the CTMC ParRep.    
However, even the embedded ParRep is far away from the linear speedup (with $100$ replicas, about $27 $ times faster than SSA).      
This sublinear speedup comes from the fact that when the number of replica is large, the acceleration is offset by the inefficient rejection sampling based dephasing procedure. 
We expect that the embedded ParRep would be more efficient if the Fleming-Viot particle processes are used for dephasing.

%--------------------------------------------------------------------------
\subsection{Reaction networks with nonlinear propensity}
In the second example, we focus on the following network from \cite{weinan2005nested},
\begin{equation}\label{eqn:nonlinear}
S_1 \rightleftharpoons S_2,\;\; S_1 \rightleftharpoons S_3, \;\; 2S_2 + S_3 \rightleftharpoons 3S_4.
\end{equation}
The propensity function and state change vector associated with each reaction is shown in Table~\ref{tab:reaction-network-nonlinear}. Note that by the law of mass action, the reactions $2S_2 + S_3 \rightleftharpoons 3S_4$ have nonlinear propensity functions.
\begin{table}[ht]
\centering
\caption{Reactions, propensity functions and state change vectors}
\begin{tabular}{l|l|l}
\hline
Reaction & Propensity function & State change vector\\
\hline\hline
$S_1 \longrightarrow S_2 $ & $\lambda_1(x) = c_1x_1$ & $\eta_1=(-1, 1, 0, 0)$\\
\hline
$S_2 \longrightarrow S_1 $  & $\lambda_2(x) = c_2x_2$  & $\eta_2 = (1, - 1, 0, 0)$ \\
\hline
$S_1 \longrightarrow S_3$  & $\lambda_3(x) = c_3 x_1$  & $\eta_3 = (-1, 0, 1, 0)$\\
\hline
$S_3 \longrightarrow S_1$  & $\lambda_4(x) = c_4 x_3$  & $\eta_4 = (1, 0, -1, 0)$\\
\hline
$2S_2 + S_3 \longrightarrow 3S_4$ & $\lambda_5(x) = c_5 x_2(x_2 - 1)x_3$ & $\eta_5 = (0, -2, -1, 3)$\\
\hline 
$3S_4 \longrightarrow 2S_2 + S_3$ & $\lambda_6(x) = c_6 x_3(x_3-1)(x_3-2)$ & $\eta_6 = (0, 2, 1, -3) $\\
\hline
\end{tabular}
\label{tab:reaction-network-nonlinear}
\end{table}

Throughout this example, we choose the initial state
$x_0 = (3, 30,30,30)$ and the reaction rate constants 
\[(c_1, c_2, c_3, c_4, c_5, c_6) = (0.1, 0.1, 0.1, 0.1, 2, 2)\]
for simulation. In this reaction network, $2S_2 + S_3 \rightleftharpoons 3S_4$ are fast reactions due to the cubic form of the propensity functions. The rest of the reactions are considered as slow reactions. 
We plot time evolution of the total propensity of reaction $5$ and $6$ versus the total propensity of reaction $1$ to $4$ in Figure~\ref{fig:nonlinear-timescale}
The timescale separation is more than $\epsilon = 10^{-4}$ as shown in the plot.
The slow observables $f_1(x) = x_2 + x_3 + x_4$ and $f_2(x) = x_1$ remain unchanged until any of the slow reactions occur.  
Therefore, the state space $E$ can be partitioned as a disjoint union of metastable sets in terms of slow observables $f_1$ and $f_2$. 
That is, $E = \cup W_{m,n}$, where $W_{m,n} = \{x \in E : f_1(x) = m, f_2(x) = n\}$.

\begin{figure}[h]\label{fig:nonlinear-timescale}
\centering
\includegraphics[scale=0.5]{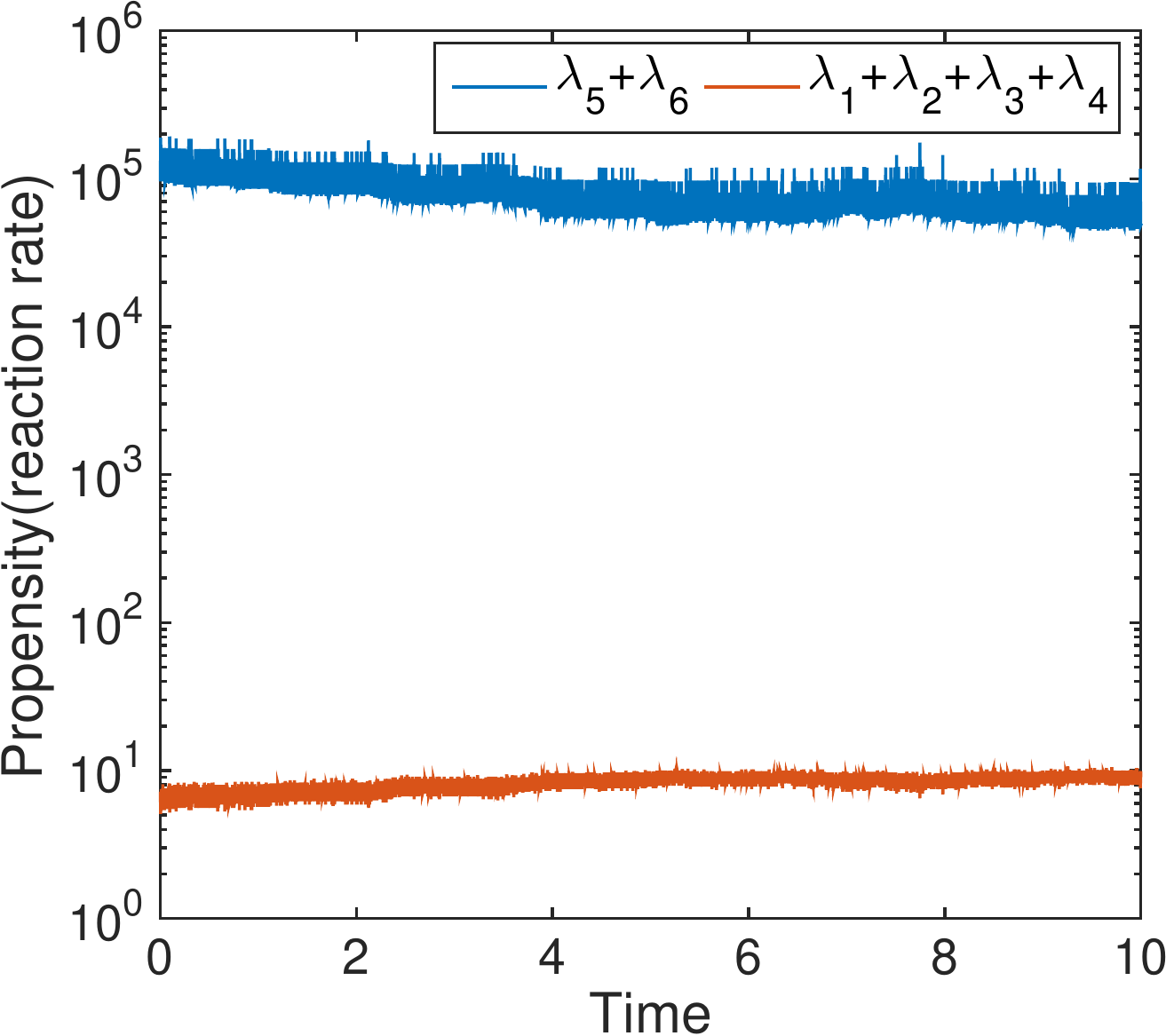}
\caption{The timescale separation between fast reactions and slow reactions. The blue curve above and the red curve below show the time evolution of $\lambda_5 + \lambda_6$ and $\lambda_1+\lambda_2+\lambda_3+\lambda_4$, respectively. It can be seen that the total propensity (reaction rate) of the last two reactions are more than $10^4$ larger than that of the first four reactions.}
\end{figure}
 
In the numerical simulation of \eqref{eqn:nonlinear}, we apply the embedded ParRep with rejection sampling and with Fleming-Viot (FV) sampling, respectively.
We are interested in the stationary average of the fast variable $\pi(x_4)$. 
The user-specified terminal time is chosen to be $T_{\mathrm{end}} = 10^3$, which is large enough for the system to be well into the stationary dynamics.
Figure~\ref{fig:nonlinear-avg} shows the estimated result (with confidence interval) for $\pi(x_4)$ with the rejection sampling based embedded ParRep (left) 
and the FV sampling based embedded ParRep (right), each with different decorrelation and dephasing thresholds.
The corresponding speedup factor is shown in Figure~\ref{fig:nonlinear-speedup}, where the left plot shows the speedup for $n_c = n_p = 20$ and the right plot shows the speedup for $n_c=n_p=60$.
It can be seen that when the decorrelation and dephasing thresholds are small (i.e., $20$), there is no performance enhancement (as shown in the left plot) when the rejection sampling is replaced by the FV sampling. 
This is consistent with our expectation that most of the replicas finish the dephasing stage after $20$ transitions (that is, no replicas escape the metastable set in $20$ transitions) and hence the FV sampling is not needed to improve the performance.  
However, when the thresholds are increased to $60$, the FV sampling based embedded ParRep outperforms the rejection sampling based embedded ParRep especially when the number of replica is large, as shown in the right plot.   
We expect that the FV sampling based ParRep would be more advantageous than the rejection sampling based ParRep when large $n_c$ and $n_p$ are needed, e.g., when the time scale separation is very large (say $\epsilon = 10^{-10}$). 

Finally, we comment that 
in many cases of stochastic reaction network model the timescales of the dynamics could change
over time. 
For instance, the last two reactions are slow if we choose $(100, 3, 3, 3)$ as the initial state in this numerical example. 
However, when the number of $S_2$ and $S_3$ increases, the last two reactions become fast.
If we still define the last two reactions as the slow reactions, then the parallel stage will be not be activated in which case the ParRep 
becomes equivalent to SSA. 
A possible remedy for this issue is to use dynamic partition of slow and fast reactions. See \cite{weinan2005nested} for a detailed discussion. 
We will deal with this issue in a separate work \cite{TWPP:2017}. 

\begin{figure}\label{fig:nonlinear-avg}
\centering
\includegraphics[scale=0.6]{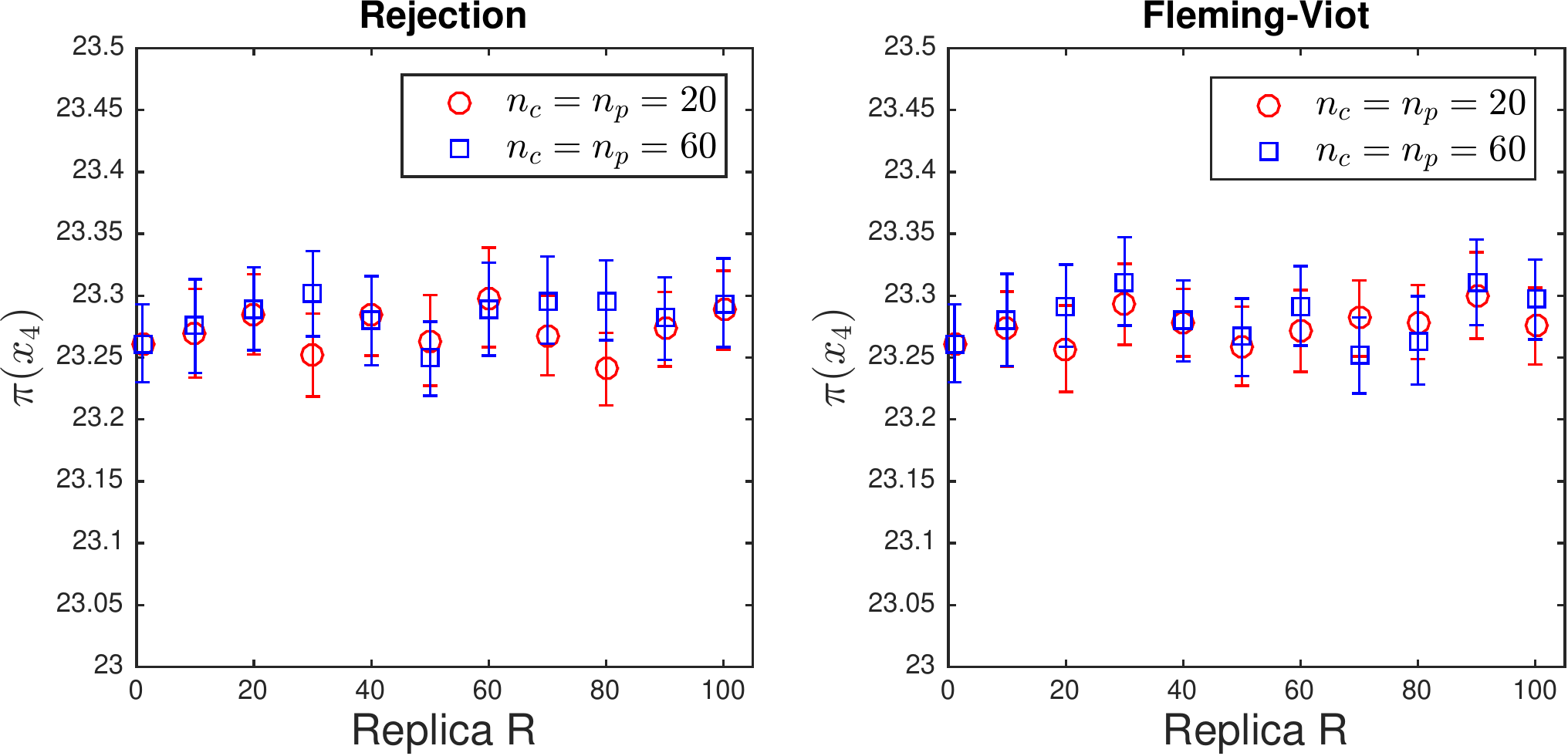}
\caption{Stationary average (with confidence interval) of the fast observable $x_4$ computed with $20$ decorrelation and dephasing steps and $60$ decorrelation and dephasing steps, respectively. Both the rejection sampling and the Fleming-Viot sampling are used for the dephasing stage.}
\end{figure}

\begin{figure}\label{fig:nonlinear-speedup}
\centering
\includegraphics[scale=0.6]{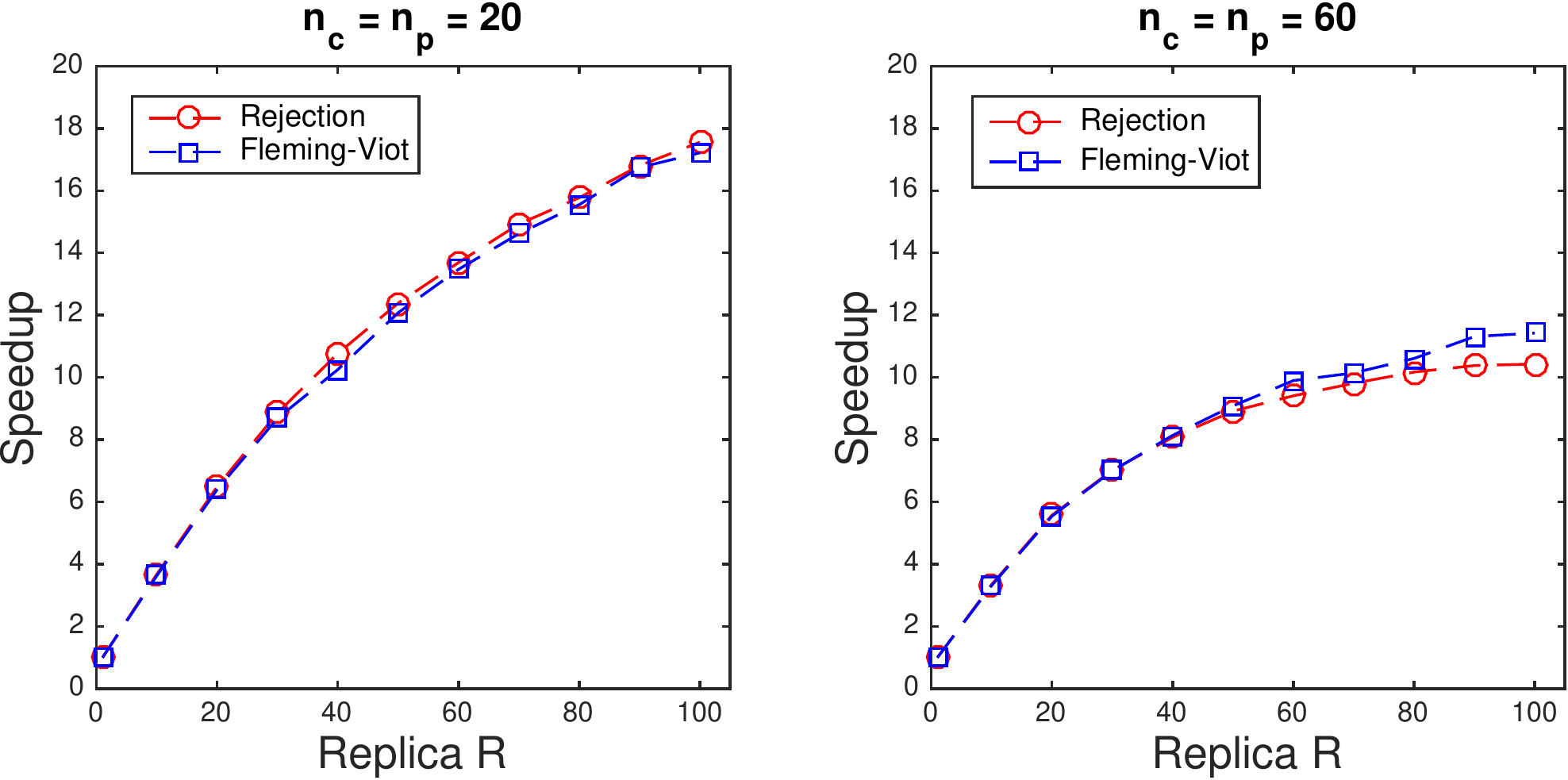}
\caption{Speedup of the ParRep with the rejection sampling based dephasing and the ParRep with Fleming-Viot sampling based dephasing.}
\end{figure}

%=====================================
\section{Conclusions}
In this paper, we propose a new method for simulating metastable CTMCs and estimating its stationary distribution with an application to stochastic reaction network models. 
The method is based on the parallel replica dynamics which first appeared in \cite{voter1998parallel}.
The ParRep method proposed here does not require the reversibility (detailed balance) of the simulated Markov chain, 
which is the necessary assumption for most accelerated algorithms for metastable dynamics simulation.
This makes the ParRep particularly well suited for stochastic reaction network model where the reversibility is not satisfied in general.    

To accelerate the estimation of stationary distribution of a metastable CTMC, our method introduces a source of error: 
we sample an approximation of the QSD of the metastable set in each decorrelation and dephasing stages. 
However, our error analysis shows that in average
the error from each ParRep cycle decays exponentially assuming the dephasing stage sampling is   
exact. Moreover, our numerical examples also suggest the consistency of the ParRep method.
The global error analysis for ParRep (i.e., the error accumulated over the entire simulation) is much more involved and will be the focus of our future work. 

The mathematical theory underlying the ParRep method predicts that we could achieve approximately linear speedup in terms of the number of replicas. 
However, due to the computation in the decorrelation and dephasing stages, the acceleration achieved in practical implementations is sub-linear.
Nevertheless, we observe a considerable performance enhancement in presented numerical examples. 
We believe further speedup is possible with a better parallel implementation of the algorithm on massively parallel clusters.   

In the numerical examples considered in this paper, we define the metastable sets in terms of the slow observable 
and assume that the partition of fast and slow reactions are fixed with time. 
However, it is quite common that the timescales of the dynamics can change
over time in many cases, especially in stochastic reaction model.
In many models the separation of time scales can change of timescales over the course of system's evolution. 
%For this type of problems, we believe that a dynamic partition in terms of fast and slow reactions could  possibly work.
Another type of 
For example, such situation
occurs in stochastic reaction networks with multiple stationary points. 
In such a case the partition of the fast and slow reactions changes when the process leaves from the neighborhood of a current stable stationary
point and move to the neighborhood of another stable stationary point. 
In this case, a different strategy (rather than fast - slow reactions) can be used to define the metastable sets. 
The ParRep method for dynamics with bistability is discussed in \cite{TWPP:2017}. 

The algorithms developed in this paper assumes that the underlying processors are synchronous. 
However, we believe both the CTMC ParRep and the embedded ParRep can be implemented in asynchronous architectures as well.
In particular, the idea for handling asynchronous processors discussed in \cite{ParRep-SDE} (Section $3$) can, in principle, be applied to the embedded ParRep as well. 
We will focus on formalizing these synchronization ideas in our future work.  
% As for the CTMC ParRep, we believe the null-event simulation technique (\cite{martinez2008synchronous}) could be used to address the asynchronous issues. 

%Finally, we comment on implementing the embedded ParRep in asynchronous architectures. 
%The embedded ParRep proposed in the paper require that all processors are to be synchronous.  
%However, similar to the  discussed in \cite{ParRep-SDE} (Section $3$), we believe that synchronous processors are not required to implement the embedded ParRep in practice.  

%===============================
\section*{Acknowledgment}
The work of T.W. has been supported by the U.S. Department of Energy, Office of Science, 
Office of Advanced Scientific Computing Research, Applied Mathematics program under the contract 
number DE-SC0010549. 
The work of P.P.  was partially supported by the DARPA project W911NF-15-2-0122.
D.A. gratefully acknowledges support from the National Science Foundation via the award NSF-DMS-1522398.

%================================

\bibliographystyle{siamplain}
\bibliography{ParRep}

\begin{thebibliography}{10}

\bibitem{ParRep-stationary}
{\sc D.~Aristoff}, {\em The parallel replica method for computing equilibrium
  averages of {M}arkov chains}, Monte Carlo Methods Appl., 21 (2015),
  pp.~255--273.

\bibitem{ParRep-Chain}
{\sc D.~Aristoff, T.~Leli{\`e}vre, and G.~Simpson}, {\em The parallel replica
  method for simulating long trajectories of markov chains}, Appl. Math. Res.
  Express., 2014 (2014), pp.~332--352.

\bibitem{binder2015generalized}
{\sc A.~Binder, T.~Leli{\`e}vre, and G.~Simpson}, {\em A generalized parallel
  replica dynamics}, J. Comput. Phys., 284 (2015), pp.~595--616.

\bibitem{blanchet2014theoretical}
{\sc J.~Blanchet, P.~Glynn, and S.~Zheng}, {\em Theoretical analysis of a
  stochastic approximation approach for computing quasi-stationary
  distributions}, arXiv preprint arXiv:1401.0364,  (2014).

\bibitem{baby-Bremaud}
{\sc P.~Br\'{e}maud}, {\em Markov Chains: {G}ibbs fields, {M}onte {C}arlo
  simulation, and Queues}, Springer, 1998.

\bibitem{QSD}
{\sc P.~Collet, S.~Mart{\'\i}nez, and J.~San~Mart{\'\i}n}, {\em
  Quasi-stationary distributions: Markov chains, diffusions and dynamical
  systems}, Springer Science \& Business Media, 2012.

\bibitem{cotter2011constrained}
{\sc S.~L. Cotter, K.~C. Zygalakis, I.~G. Kevrekidis, and R.~Erban}, {\em A
  constrained approach to multiscale stochastic simulation of chemically
  reacting systems}, J. Chem. Phys., 135 (2011), p.~094102.

\bibitem{del2004particle}
{\sc P.~Del~Moral and A.~Doucet}, {\em Particle motions in absorbing medium
  with hard and soft obstacles}, Stoch. Anal. Appl., 22 (2004), pp.~1175--1207.

\bibitem{dupuis2012infinite}
{\sc P.~Dupuis, Y.~Liu, N.~Plattner, and J.~D. Doll}, {\em On the infinite
  swapping limit for parallel tempering}, Multiscale Model Simul., 10 (2012),
  pp.~986--1022.

\bibitem{earl2005parallel}
{\sc D.~J. Earl and M.~W. Deem}, {\em Parallel tempering: Theory, applications,
  and new perspectives}, Phys. Chem. Chem. Phys., 7 (2005), pp.~3910--3916.

\bibitem{ferrari2007quasi}
{\sc P.~A. Ferrari and N.~Maric}, {\em Quasi stationary distributions and
  fleming-viot processes in countable spaces}, Electron. J. Probab., 12 (2007),
  pp.~684--702.

\bibitem{geyer1991markov}
{\sc C.~J. Geyer}, {\em Markov chain monte carlo maximum likelihood}, in
  Computing Science and Statistics: Proceedings of the 23rd Symposium on the
  Interface,  (1991), pp.~156--163.

\bibitem{gillespie1977exact}
{\sc D.~T. Gillespie}, {\em Exact stochastic simulation of coupled chemical
  reactions}, J. Phys. Chem., 81 (1977), pp.~2340--2361.

\bibitem{TTS}
{\sc A.~Hashemi, M.~Nunez, P.~Plech{\'a}{\v{c}}, and D.~G. Vlachos}, {\em
  Stochastic averaging and sensitivity analysis for two scale reaction
  networks}, J. Chem. Phys., 144 (2016), p.~074104.

\bibitem{ParRep-SDE}
{\sc C.~Le~Bris, T.~Lelievre, M.~Luskin, and D.~Perez}, {\em A mathematical
  formalization of the parallel replica dynamics}, Monte Carlo Methods Appl.,
  18 (2012), pp.~119--146.

\bibitem{perez2015parallel}
{\sc D.~Perez, B.~P. Uberuaga, and A.~F. Voter}, {\em The parallel replica
  dynamics method--coming of age}, Comput. Mater. Sci., 100 (2015),
  pp.~90--103.

\bibitem{plattner2011infinite}
{\sc N.~Plattner, J.~Doll, P.~Dupuis, H.~Wang, Y.~Liu, and J.~Gubernatis}, {\em
  An infinite swapping approach to the rare-event sampling problem}, J. Chem.
  Phys., 135 (2011), p.~134111.

\bibitem{rathinam2003stiffness}
{\sc M.~Rathinam, L.~R. Petzold, Y.~Cao, and D.~T. Gillespie}, {\em Stiffness
  in stochastic chemically reacting systems: The implicit tau-leaping method},
  J. Chem. Phys., 119 (2003), pp.~12784--12794.

\bibitem{rubinstein}
{\sc R.~Y. Rubinstein and D.~P. Kroese}, {\em Simulation and the Monte Carlo
  method}, vol.~707, John Wiley \& Sons, 2011.

\bibitem{schutte}
{\sc C.~Sch\"{u}tte and M.~Sarich}, {\em Metastability and Markov State Models
  in Molecular Dynamics}, Courant Lecture Notes, 2013.

\bibitem{non-negativeMC}
{\sc E.~Seneta}, {\em Non-negative matrices and Markov chains}, Springer
  Science \& Business Media, 2006.

\bibitem{voter1998parallel}
{\sc A.~F. Voter}, {\em Parallel replica method for dynamics of infrequent
  events}, Phys. Rev. B, 57 (1998), p.~R13985.

\bibitem{TWPP:2017}
{\sc T.~Wang and P.~Plech{\'a}{\v{c}}}, {\em Parallel replica dynamics method
  for bistable stochastic reaction networks: simulation and sensitivity
  analysis}, J. Chem. Phys., 147 (2017), p.~234110.

\bibitem{weinan2005nested}
{\sc E.~Weinan, D.~Liu, and E.~Vanden-Eijnden}, {\em Nested stochastic
  simulation algorithm for chemical kinetic systems with disparate rates}, J.
  Chem. Phys., 123 (2005), p.~194107.

\end{thebibliography}
\end{document}